\documentclass[preprint]{elsarticle} %review  preprint
 
\usepackage[hidelinks]{hyperref}
 \usepackage{setspace} 
\usepackage{epsfig,amssymb,latexsym}
\usepackage{amsfonts,psfrag,amsmath,mathrsfs,amsthm,color,empheq}
\usepackage{graphicx}
\usepackage{textcomp}
\usepackage{multirow}
\usepackage{enumerate}
\usepackage{caption} 
\usepackage{slashbox} % used in the table
\usepackage{booktabs}
\usepackage{float}
\usepackage{tikz} % for plot 
\usetikzlibrary{decorations.pathreplacing}
\journal{ Computers \& Mathematics with Applications}

\newtheorem{theorem}{Theorem}[section]

\newcommand{\IR}{{\mathbb{R}}}

\newcommand\figcaption{\def\@captype{figure}\caption}
\newcommand\tabcaption{\def\@captype{table}\caption}

\graphicspath{{./Figures/}}
\usepackage[margin=1in]{geometry}
\newcommand{\tn}[1]{\ \textnormal{#1}\ } 
\newcommand{\bmt}{\left[ \begin{array}{ccccccccccccccccccccccccccccccccccccc}}
\newcommand{\emt}{\end{array}\right]}
\newcommand{\bean}{\begin{eqnarray*}}
\newcommand{\eean}{\end{eqnarray*}}
\newcommand{\bea}{\begin{eqnarray}}
\newcommand{\eea}{\end{eqnarray}}
\newcommand{\eq}{\begin{equation}\begin{array}{lllllllll}}
\newcommand{\ee}{\end{array}\end{equation}}
\newcommand{\eqn}{\begin{equation*}\begin{array}{lllllllll}}
\newcommand{\een}{\end{array}\end{equation*}}

\def\x{\xi}

\def\b{\beta}
\def\g{\gamma}

\def\p{\phi}

\def\f{{\bf f}}

\begin{document}

\begin{frontmatter}

\title{Efficient Time Domain Decomposition Algorithms for\\ Parabolic PDE-Constrained Optimization Problems \tnoteref{t1}}
\tnotetext[t1]{This project was supported in part by the grant NSF-DMS 1522672 of the United States.}

%% Group authors per affiliation:
\author[Liu]{Jun Liu\corref{mycorrespondingauthor}}
\ead{juliu@siue.edu}
\author[Wang]{Zhu Wang}
 \ead{wangzhu@math.sc.edu}
%\cortext[cor1]{Corresponding author}
%% or include affiliations in footnotes:
\address[Liu]{Department of Mathematics and Statistics, Southern Illinois University Edwardsville, Edwardsville, IL 62025, USA.}
\address[Wang]{Department of Mathematics, University of South Carolina,
Columbia, SC 29208, USA.}
\cortext[mycorrespondingauthor]{Corresponding author}

\begin{abstract}
Optimization with time-dependent partial differential equations (PDEs) as constraints {appears} in many science and engineering applications. 
The associated first-order necessary optimality system consists of one forward and one backward time-dependent PDE coupled with optimality conditions. 
An optimization process by using the one-shot method determines the optimal control, state and adjoint state at once, 
with the cost of solving a large scale, fully discrete optimality system. 
Hence, such {a} one-shot method could easily become computationally prohibitive when the time span is long or time step is small. 
To overcome this difficulty, we propose several time domain decomposition algorithms for improving the {computational efficiency of the one-shot method}. 
In these algorithms, the optimality system is split into many small subsystems over a much smaller time interval, which are coupled by appropriate continuity matching conditions. 
Both one-level and two-level multiplicative and additive Schwarz algorithms are developed for iteratively solving the decomposed subsystems in parallel.
In particular, the convergence of the one-level, non-overlapping algorithms is proved.
The effectiveness of our proposed algorithms is demonstrated by both 1D and 2D numerical experiments,
where the developed two-level algorithms {show convergence rates that are scalable} with respect to the number of subdomains.
\end{abstract}


\end{frontmatter}

%\linenumbers

%\input{introduction}
%=============================================================
\section{Introduction}\label{s_introduction}
%=============================================================
Many scientific and engineering applications involve the optimization of systems governed by time-dependent partial differential equations (PDEs) \cite{Lorenz2003}, such as
the control and optimization of flows \cite{gunzburger2003perspectives}, the optimal {quenching for solidification in molds},
%heating and cooling of minimizing the energy consumption subject to the achievement of a prescribed temperature distribution,
and the {design of airfoils to minimize their} drag subject to a minimal lift.
%problems combine the computational fluid dynamics methodologies with optimization strategies. 
One way {to treat} such constrained optimization problems is to construct an unconstrained optimization problem through 
the Lagrange multiplier method \cite{hinze2008optimization}. 
By employing the optimize-then-discretize approach,
the corresponding first-order necessary optimality system is often given by a forward-and-backward in time PDE system. In such a PDE system, the state equations are posed forward with initial conditions, the adjoint equations are posed backward with terminal conditions at the final time, and they are coupled by optimality conditions.

% {why one-shot method?}
The optimal state and control variables can be directly determined from the discrete optimality system. 
This approach is referred to as the one-shot method. 
However, since the system has twice as many unknowns as the state equation and is fully discrete in both space and time, the one-shot method needs to solve a large scale system. 
Therefore, sensitivity- and adjoint-based iterative optimization methods have been developed \cite{gunzburger2003perspectives}, which solve the state equations and adjoint state equations alternatively so that the computational burden at each optimization iteration is affordable. 
Nevertheless, the one-shot method is still very appealing, in particular, for convex problems, 
because a single solving process would be sufficient to obtain the optimal states, adjoint states, and optimal control all together, which does not require any  {inherently sequential} optimization iterations. 
In order to improve its overall efficiency in solving large-scale PDE-constrained optimization, 
many efficient direct/iterative linear/nonlinear solvers have been proposed in the framework of one-shot method, 
see for example \cite{Biros_2005a,Biros_2005b,Prudencio_2006,Barker_2010,Cai2012,Borzi2012,Axelsson_2016,Yang_2016} and the reference therein. 
It was shown in \cite{Borzi2012} that, when the multigrid method is used, coarsening in space only (semi-coarsening) yields better convergence than coarsening in both space and time. 
Thus, it is quite natural to investigate numerical remedies, such as time parallel computing, that could relieve the high computational cost due to the time discretization. 
% of a long time interval or with a short time step

Parallel-in-time methods have been investigated for evolutionary problems over the last four decades. 
Their central idea is to distribute a tremendous computational task into many small parts, which can be executed by multiple processors simultaneously.  
Since Nievergelt proposed the first time decomposition algorithm for finding the parallel solutions of evolutionary differential equations {\cite{nievergelt1964parallel}},
the methods of time-parallel time integration have been extensively expanded. 
%which can be categorized into four classes: 
%multiple shooting methods, 
%domain decomposition and waveform relaxation methods, 
%space-time multigrid methods, 
%and direct time parallel methods. 
The reader is referred to \cite{gander50yrs} for a survey of these methods.
For the forward-backward PDEs, a time parallel iterative method was studied in \cite{Heinkenschloss2005}, which is based on a multiple shooting reformulation of the linear quadratic optimal control problems. 
A domain decomposition in time algorithm was investigated in \cite{barkerdomain}, which follows the discretize-then-optimize approach and employs the additive Schwarz preconditioner in time based on the discrete system. 
More recently, some preliminary convergence analysis {of} Schwarz methods in time for linear parabolic control problems was conducted in \cite{Gander_2016,Felix_2016}, 
where an appealing mesh-independent convergence rate is proved under some restrictive technical assumptions on the discretized spatial operators. 
We note that domain decomposition algorithms in space have been well established (\cite{smith1996domain,toselli2005domain,mathew2008domain,Cai2006,Heinkenschloss_2007,Dolean2015}) in solving large systems. 
But time domain decomposition methods have been less thoroughly investigated in the context of forward-and-backward 
optimality system. 

Therefore, we put forth a time domain decomposition framework for the forward-and-backward PDEs {appearing} in optimality systems.   
The proposed approach {has its} roots in an equivalent hybrid formulation of the continuous system. 
Four different types of iterative domain decomposition algorithms are designed to compute numerical solutions by multiple processors in parallel: 
multiplicative and additive Schwarz algorithms with or without overlap. 
 {Since the one-level domain decomposition algorithms would require more iterations for convergence as the number of subdomains increases, to obtain better performance with high degree of parallelism, it is often necessary to use two-level domain decomposition algorithms equipped with a coarse grid correction procedure \cite{smith1996domain,quarteroni1999domain,toselli2005domain,Yang_2016,Deng_2015}. 
Therefore, both one-level and two-level algorithms are numerically investigated in this paper. 
%In this context, several two-level space-time domain decomposition preconditioners were proposed for solving flow control \cite{Yang_2016} and inverse source problems \cite{Deng_2015}, where however only overlapping algorithms as preconditioners are considered. 
In particular, two-level domain decomposition algorithms 
based on some recently developed coarse space correction techniques are implemented \cite{Martin2012,Gander2014}. }
Although it has been pointed out in \cite{barkerdomain,Felix_2016} that two-level algorithms are needed to 
achieve better scalable convergence rates than one-level algorithms, 
there is very few work discussing such two-level time domain decomposition algorithms in the setting of 
forward-and-backward optimality system.
For instance, some two-level space-time parallel domain decomposition preconditioners {were very recently} proposed in \cite{Deng_2015,Yang_2016}, 
where the full discretizations are based on first-order accurate backward Euler scheme in time and the convergence analysis of the algorithms was not given.
{In contrast to the aforementioned} methods, our proposed time domain decomposition algorithms are based on a second-order accurate leapfrog scheme in time
and are proved to be convergent as stand-alone solvers.
Moreover, the construction of our coarse grid correction is very different from the geometrical {arguments} used in \cite{Cai2012,Deng_2015,Yang_2016,Kong_2016}.
The reported comprehensive numerical comparison of four types of domain decomposition algorithms (both one-level and two-level) as stand-alone solvers and preconditioners is another major contribution of this work. 
In particular, compared to overlapping methods, two-level non-overlapping domain decomposition methods \cite{Carvalho_2001,Giraud_2003} have been rarely discussed in the literature. 
%due to its relatively less robust convergence.
Our current work contributes to extend the applicability of non-overlapping methods. 
{It is worth mentioning that, 
although we solve each forward-backward optimality subsystem with either sparse direct solver or existing multigrid solver in this paper, other iterative methods could be applied to find solutions in each time subinterval. Implementation techniques such as checkpointing \cite{griewank2000algorithm} could be further used to trade memory for run-time by reversing only parts of a program at a time.}

The rest of the paper is organized as follows. 
In Section \ref{sec: opt}, we consider a typical linear parabolic PDE constrained optimization problem and present its 
forward-and-backward optimality PDE system and the corresponding finite difference discretization. 
In Section \ref{sec: tdd}, a hybrid formulation with two subdomains is proved to motivate our time domain decomposition algorithms. 
Based on the developed hybrid formulation, four different domain decomposition algorithms are proposed and some convergence analyses are given in Section \ref{sec: itm};  
 {two-level algorithms are studied in Section \ref{sec: two}.}
In Section \ref{sec: num}, some comprehensive numerical experiments are conducted to demonstrate the effectiveness of our proposed domain decomposition algorithms
as stand-alone solvers and preconditioners, respectively. 
Finally, some  {concluding} remarks are drawn in the last section.
%\input{model}
%=============================================================
\section{Parabolic PDE-Constrained Optimization and its Discretization}\label{sec: opt}
%=============================================================
In this section, we exemplify our proposed algorithms through discussing a typical distributed control problem.
Let $\Omega$ be a bounded subset of $\IR^d$ ($1\le d\le 3$) with Lipschitz boundary $\Gamma:=\partial\Omega$.
{To facilitate the use of finite difference discretizations, we mainly choose the unit square domain, i.e., $\Omega=(0,1)^2$.
However, our following proposed domain decomposition algorithms are directly applicable to domains with general geometric shape, 
if appropriate discretizations, such as finite element method, are used in space. 
}
Define $Q=\Omega\times (0,T)$ and $\Sigma=\Gamma\times (0,T)$ for a given finite period of time $T>0$. 
We consider the {following} unconstrained optimal control problem \cite{Lions1971,fredi2010optimal} for minimizing a tracking-type quadratic cost functional
\eq \label{goal}
\quad J(y,u)=\frac{1}{2} \|y-g\|^2_{L^2(Q)}+\frac{\gamma}{2}\|u\|^2_{L^2(Q)}
\ee
subject to a linear parabolic PDE system
\eq \label{state}
\left\{\begin{array}{rlr}
-\partial_t y + \Delta y=&f+u\ &\tn{in}\ Q, \\
y =&0\ &\tn{on}\ \Sigma,\\
y(\cdot,0)=&y_0\ &\tn{in}\ \Omega,
\end{array}\right.
\ee
where $g\in L^2(Q)$ is the desired tracking trajectory, $u\in U:=L^2(Q)$ is the distributed control function, 
$\gamma>0$ represents either the weight on the cost of control or the Tikhonov regularization parameter,
$f\in L^2(Q)$ is the source term, and the initial condition $y_0\in H_0^1(\Omega)$. 
For simplicity, we also assume no constraints on the control $u$. 
The existence and uniqueness of the solution to the above problem 
is well established \cite{Lions1971,fredi2010optimal}.

By defining a Lagrange functional and making use of the strict convexity of the objective functional $J$, 
the optimal solution pair $(y, u)$ to (\ref{goal})-(\ref{state}) can be uniquely determined by the first-order necessary condition. It leads to the following PDE system:
%
%Hence, the optimal state $y$  and adjoint state $p$ are the solution to the optimality PDE system
\eq \label{opt1A}
\left\{\begin{array}{rl}
-\partial_t y + \Delta y-p/\gamma=&f\ \tn{in} Q,\quad \\
y=&0\ \tn{on} \Sigma,\quad
y(\cdot,0)=y_0 \tn{in}\ \Omega,\\
\partial_t p+ \Delta p+ y=& g\ \tn{in} Q,\quad \\
p=&0\ \tn{on} \Sigma, \quad
p(\cdot,T)=0 \tn{in}\ \Omega,
\end{array}\right.
\ee
where the state $y$ evolves forward in time and adjoint state $p$ marches backward.
The optimal control $u$ can be derived from the optimality condition $\gamma u=p$ in $Q$.
The fact that the state $y$ and the adjoint state $p$ are marching in opposite orientations represents the main challenge for solving (\ref{opt1A}) numerically, 
since we have to resolve all time steps simultaneously in computation, which inevitably results in a huge sparse system of linear equations upon a full discretization.

Following our previous work \cite{LX2015a,Liu_2016}, we use a second-order leapfrog central finite difference scheme to obtain a full discretization of the optimality system (\ref{opt1A}).
Here, we briefly describe the scheme for completeness and refer interested readers to \cite{LX2015a} for more details. 
The time interval $[0,T]$ is uniformly divided into $N$ sub-intervals by the points 
$0=t_0<t_1<\cdots<t_N=T$ and time step size $\tau=T/N$. 
The space domain $\Omega=(0,1)^2$ is partitioned uniformly by grid points $(\xi_i, \zeta_j)$ with $0\leq i, j \leq M$, where $\xi_i=ih$, $\zeta_j=jh$, and $h=1/M$.  
%$0=\xi_0<\xi_1<\cdots<\xi_{M}=1$
%and $0=\zeta_0<\zeta_1<\cdots<\zeta_{M}=1$,
%with $h=\xi_i-\xi_{i-1}=\zeta_j-\zeta_{j-1}=1/M$. 
Denote the partitioned space domain by $\Omega_h$. 
For any grid function $Z(\cdot,\cdot)$ on $\Omega_h$, we define the discrete gradient ($\nabla_h$)
 {by the forward difference and 
%$$
%\nabla_h Z =\bigg(\frac{Z_{i,j}-Z_{i-1,j}}{h},
%\frac{Z_{i,j}-Z_{i,j-1}}{h}\bigg)_{i=1,j=1}^{M,M}
%$$
define the discrete Laplacian ($\Delta_h$) by a 5-point stencil central difference.}
%(based on a central difference 5-point stencil)
%\begin{align} \label{Laplacian}
%(\Delta_h Z)_{ij}=
%\frac{Z_{i-1,j}-2Z_{i,j}+Z_{i+1,j}}{h^2}
%+\frac{Z_{i,j-1}-2Z_{i,j}+Z_{i,j+1}}{h^2} ~,\ 1\le i,j\le M-1.
%\end{align}
The full discretization yields  
\begin{align}
&-\frac{Y^{n+1}-Y^{n-1}}{2\tau}+\Delta_h Y^n-P^n/\gamma=F^n,\quad n=1,2,\cdots,N-2
,  \label{FDEq1}\\
&\frac{P^{n+1}-P^{n-1}}{2\tau}+\Delta_h P^n+   Y^n=G^n,\quad n=2,3,\cdots,N-1
,  \label{FDEq2}
\end{align}
where $(Y^n)_{ij}$ and $(P^n)_{ij}$ are the discrete approximation
%$Y^n=(Y^n_{ij})_{i=1,j=1}^{M-1,M-1}$ and $P^n=(P^n_{ij})_{i=1,j=1}^{M-1,M-1}$ with 
%$Y^n_{ij}$ and $P^n_{ij}$ being the discrete approximation
of $y(\xi_i,\zeta_j,t_n)$ and $p(\xi_i,\zeta_j,t_n)$, respectively, for $i, j= 1, \ldots M-1$. 
Similar notations are used for $F^n$ and $G^n$.
 {The Dirichlet boundary conditions are imposed at the grid points with $i=1,M-1$ or $j=1,M-1$ through the discrete Laplacian terms.}
% {In actual computation, we will eliminate the known Dirichlet boundary grid points in space by shifting them to the right-hand-side.}
Also notice that $Y^0$ and $P^N$ are obtained directly from the given initial conditions $y(\cdot,0)=y_0$ and $p(\cdot,T)=0$.
To close the above linear system, we impose two additional equations 
derived from approximating \eqref{opt1A} at the last time step (aligned with $n=N-1$ in (\ref{FDEq1}) and $n=1$ in (\ref{FDEq2}) ) by a second-order, one-sided BDF2 scheme \cite{LeVeque2007}, i.e., 
\begin{align}
-\frac{ {-Y^{N-3}+4Y^{N-2}-3Y^{N-1}} }{2\tau}+\Delta_h Y^{N-1}-P^{N-1}/\gamma&=F^{N-1} ,
\label{FDEq12}\\
\frac{3P^{1}- 4P^{2}+P^3}{ 2\tau}+\Delta_h P^1+   Y^1&=G^1
.    \label{FDEq22}
\end{align}
For the purpose of obtaining better performance in our domain decomposition algorithms, we have slightly modified the original scheme in \cite{LX2015a}
by {no longer} treating $Y^N$ and $P^0$ as unknowns. 
Otherwise, according to our numerical tests, the given initial conditions $y(\cdot,0)=y_0$ and $p(\cdot,T)=0$ cannot be easily shifted to the right-hand-side, which seems to introduce large approximation errors on the subdomain interfaces and hence {causes} deteriorated convergence rates for the 2-level algorithms based on coarse grid residual correction.

The above scheme (\ref{FDEq1}-\ref{FDEq22}) in a matrix form is 
%in a two-by-two block structured non-symmetric infinite sparse linear system
a two-by-two block, structured, sparse linear system, whose coefficient matrix is non-symmetric and indefinite, as follows.
\begin{align} \label{linsys}
 L_h w_h:=\bmt
 A_h&-I_h/\gamma\\
I_h&D_h
\emt 
 \bmt y_h\\ p_h \emt=\bmt  f_h\\ g_h \emt =:b_h,
\end{align}
where
\begin{align*} 
A_h&=\bmt
  \Delta_h &-{I}/{2\tau} & \cdots &0 &0&0\\
{I}/{2\tau} &  \Delta_h & -{I}/{2\tau} & \cdots &0&0\\
0&{I}/{2\tau} &  \Delta_h & -{I}/{2\tau} & \cdots &0\\
0& 0&\ddots &\ddots &\ddots &0\\
0& 0&\ddots &\ddots &\ddots &0\\
 0&0&\cdots &{I}/{2\tau} & \Delta_h & -{I}/{2\tau}\\
0&0&\cdots & -I/2\tau & {4I}/{2\tau} &( \Delta_h-{3I}/{2\tau}) 
\emt,\\ 
D_h&=\bmt
(\Delta_h-{3I}/{2\tau})&{4I}/{2\tau}&{-I}/{2\tau} &0&\cdots &0 \\
-{I}/{2\tau} &  \Delta_h & {I}/{2\tau} &0& \cdots &0 \\
0&-{I}/{2\tau} &  \Delta_h & {I}/{2\tau} &\ddots & \vdots \\
0& 0&\ddots &\ddots &\ddots &0\\
0&0&\ddots &-{I}/{2\tau} &\Delta_h &{I}/{2\tau} \\
0& 0&\cdots & 0&-{I}/{2\tau} & \Delta_h 
\emt, \\
f_h&=\bmt  f^1- {y_{0,h}}/{2\tau} \\f^2 \\ \vdots \\ f^{N-1} \emt,
g_h=\bmt g^1 \\g^2 \\ \vdots \\ g^{N-1} \emt,
y_h=\bmt  y^1 \\y^2 \\ \vdots \\ y^{N-1} \emt,\quad \mbox{and}\quad
p_h=\bmt  p^1 \\p^2 \\ \vdots \\ p^{N-1}  \emt.
\end{align*} 
% {Check $Y^0$ and $P^N$ because they denote matrices. We need to be consistent in the notations: We use $Y$ and $P$ to denote matrices; $y$ and $p$ for both variables and vectors? }
Here $I$ and $I_h$ are identity matrices of appropriate size
% {$\Delta_h$ denotes the matrix representation of the corresponding discrete Laplacian operator, 
and the vectors $y_{0,h}$, $f^n$, $g^n$, $y^n$ and $p^n$ are the lexicographic ordering (vectorization) of the 
corresponding function approximations on spatial grid points (in matrix form) at time step $t_n$.
Notice that the initial conditions $y(\cdot,0)=y_0$ and $p(\cdot,T)=0$ are already shifted to the right-hand-side,
i.e., the first term of $f_h$ and the last term of $g_h$.
It is not difficult to see that the size of the above system (\ref{linsys}) will become prohibitively large when $N$ or $T$ gets large enough,
which motives our proposed time domain decomposition algorithms in this paper.
In some sense,
the proposed domain decomposition algorithms in time complements very well with the developed semi-coarsening multigrid algorithm in space \cite{LX2015a}, as shown in the subsection \ref{num2D}.

%\input{algorithm}
%=============================================================
\section{A Hybrid Continuous Formulation of Optimality System with Two Subdomains }\label{sec: tdd}
%=============================================================
%In this section, we assume the time interval is split into two subdomains and introduce a hybrid formulation of the optimality system \eqref{opt1A}. 
In this section, we assume the time interval $[0,T]$ is decomposed into two subdomains, $R_1= (0, T_1^{r})$ and $R_2= (T_2^ l, T)$, where $T_1^{r} \geq T_2^ l$. 
% {add a schematic for the decompositions of $Q$}
% {Define the partition of unity here: $\chi_i$ for $i=1, 2$. }
When $T_1^{r} = T_2^ l$, there is no overlapping between two subdomains; 
when $T_1^{r}>T_2^ l$, there is an overlapping region. 
Let $y_i$, $p_i$ be the local solutions on each subdomain $Q_i= \Omega \times R_i$, for $i=1, 2$. 
We introduce a hybrid formulation that is the restriction of the optimality system \eqref{opt1A} on each subdomain $Q_i$ together with the continuity matching conditions on the interfaces or regions of overlap between adjacent subdomains.
%Due to the requirement of continuity in time of the local solutions $y_i$ and $p_i$ across adjacent time subdomains,
%We have $y_i= y_j$ and $p_i = p_j$ on $\partial \Omega_i \cap \Omega_j$ or $\partial \Omega_i^* \cap \Omega_j^*$. 
We will first show that the hybrid formulation is equivalent to the original system, i.e., the local solutions of the hybrid problem coincide the global solution restricted on each subdomain,  
which is not surprising but very fundamental to justify and analyze the convergence of any domain decomposition algorithms.
In the next section, we will show that time domain decomposition algorithms
based on this hybrid formulation lead to a solution convergent to the solution of original system.

\begin{theorem} \label{thm1}
Let $y(x, t), p(x, t)\in W(0, T; Q):= \{y\in H^{1, 0}(Q), \partial_t y\in L^2\left(0, T; H^1(\Omega)^*\right)\}$ be the solutions to optimality system \eqref{opt1A}. 
Suppose $y_1(x, t)$, $p_1(x, t)$ and $y_2(x, t)$, $p_2(x, t)$ 
solve the following hybrid systems
\eq \label{tdd_2s_1}
\left\{\begin{array}{rl}
-\partial_t y_1 + \Delta y_1-p_1/\gamma=&f\ \tn{in} Q_1,\quad \\
y_1=&0\ \tn{on} \Sigma,\quad
y_1(\cdot,0)=y_0 \tn{in}\ \Omega,\\
\partial_t p_1+ \Delta p_1+ y_1=& g\ \tn{in} Q_1,\quad \\
p_1=&0\ \tn{on} \Sigma, \quad
p_1(\cdot,T_1^{r})=p_2(\cdot,T_1^{r}) \tn{in}\ \Omega,
\end{array}\right.
\ee
and
\eq \label{tdd_2s_2}
\left\{\begin{array}{rl}
-\partial_t y_2 + \Delta y_2 - p_2/\gamma=&f\ \tn{in} Q_2,\quad \\
y_2 =&0\ \tn{on} \Sigma,\quad
y_2(\cdot,T_2^ l)=y_1(\cdot,T_2^ l) \tn{in}\ \Omega,\\
\partial_t p_2+ \Delta p_2+ y_2=& g\ \tn{in} Q_2,\quad \\
p_2=&0\ \tn{on} \Sigma, \quad
p_2(\cdot,T)=0 \tn{in}\ \Omega,
\end{array}\right.
\ee
which are coupled through the continuity matching conditions.
Then, there hold
\eq \label{tdd_2s_same}
\left\{\begin{array}{rl}
y_1=y \tn{and} p_1 = p&  \tn{in} Q_1,\quad \\
y_2=y \tn{and} p_2 = p & \tn{in} Q_2. \quad 
\end{array}\right.
\ee
\end{theorem}

\begin{proof}
If $y$ and $p$ are solutions of the optimality system \eqref{opt1A}, $y_i= y$  and $p_i= p$ on $Q_i$, then $y_1$, $y_2$ and $p_1$, $p_2$ will obviously satisfy the hybrid formulation \eqref{tdd_2s_1}-\eqref{tdd_2s_2} by the construction. 

%By Neitzel et al. Theorem 2.6, we know the optimality system control $u$ is optimal iff $(y, p)\in W(0, T)$ and $(y, u, p)$ is the weak solution of (2.7), (2.10), and the equation below (2.10). 
%
%Based on Thm 3.18 on page 158 of Fredi's book, for non-overlapping case, we can show the hybrid formulation is equivalent to the original system. 

To prove the converse, we suppose that $y_1$, $y_2$ and $p_1$, $p_2$ solve the hybrid formulation \eqref{tdd_2s_1}-\eqref{tdd_2s_2} and will show that $y$ and $p$ satisfy \eqref{opt1A} if $y= y_i$ and $p= p_i$ on $Q_i$. 
In particular, we need to show that $(y, p)\in W(0, T; Q)$ satisfies the weak formulation of the optimality system \eqref{opt1A}: 
\begin{eqnarray}
\iint_Qy v_t - \iint_Q \nabla y \cdot \nabla v - \frac{1}{\gamma}\iint_Q p v &=&  \iint_Q fv + \int_{\Omega} y(\cdot, T) v(\cdot, T)  \nonumber \\ 
&-& \int_{\Omega} y(\cdot, 0) v(\cdot, 0),\quad \forall\, v\in W^{1, 1}_2(Q). \label{eq:hyb1} \\
-\iint_Q p w_t - \iint_Q \nabla p \cdot \nabla w + \iint_Q y w &=&  \iint_Q g w + \int_{\Omega} p(\cdot, 0) w(\cdot, 0)\nonumber \\  
&-& \int_{\Omega} p(\cdot, T) w(\cdot, T),\quad \forall\, w\in W^{1, 1}_2(Q). \label{eq:hyb2}
\end{eqnarray}

(I) We first consider the non-overlapping decomposition in time domain, i.e., $T_1^{r} = T_2^ l$.  
By testing the equations of $y_1 \in W(0, T_1^r; Q_1)$ and $y_2\in W(T_2^ l, T; Q_2)$, \eqref{tdd_2s_1} and \eqref{tdd_2s_2}, by $v$ respectively, we obtain 
\begin{eqnarray}
\iint_{Q_1} y_1 v_t - \iint_{Q_1} \nabla y_1 \cdot \nabla v - \frac{1}{\gamma}\iint_{Q_1} p_1 v &=&  \iint_{Q_1} fv + \int_{\Omega} y_1(\cdot, T_1^{r}) v(\cdot, T_1^{r})  \nonumber \\ 
&-& \int_{\Omega} y_1(\cdot, 0) v(\cdot, 0),\quad \forall\, v\in W^{1, 1}_2(Q), \label{eq:hyb1_a}\\ 
\iint_{Q_2} y_2 v_t - \iint_{Q_2} \nabla y_2 \cdot \nabla v - \frac{1}{\gamma}\iint_{Q_2} p_2 v &=&  \iint_{Q_2} fv + \int_{\Omega} y_2(\cdot, T) v(\cdot, T)  \nonumber \\ 
&-& \int_{\Omega} y_2(\cdot, T_1^{r}) v(\cdot, T_1^{r}),\,\, \forall\, v\in W^{1, 1}_2(Q). \label{eq:hyb1_b}
\end{eqnarray}
Considering $y= y_1$ in $Q_1$, $y= y_2$ in $Q_2$, $y_1(\cdot, T_1^{r}) = y_2(\cdot, T_1^{r})$ and summing \eqref{eq:hyb1_a} and \eqref{eq:hyb1_b} together, we get a weak formulation for equations of $y$ that is equivalent to \eqref{eq:hyb1}.
An analogous argument leads to, given $p= p_1$ in $Q_1$ and $p= p_2$ in $Q_2$, the weak formulation for equations of $p$ in the coupled system is equivalent to \eqref{eq:hyb2}. 

%\begin{eqnarray}
%\iint_Qy v_t - \iint_Q \nabla y \cdot \nabla v - \frac{1}{\gamma}\iint_Q p v &=&  \iint_Q fv + \int_{\Omega} y(\cdot, T) v(\cdot, T)  \nonumber \\ 
%&-& \int_{\Omega} y(\cdot, 0) v(\cdot, 0),\quad \forall\, v\in L^2(0, T; V), \\ \nonumber
%-\iint_Q p w_t - \iint_Q \nabla p \cdot \nabla w + \iint_Q y w &=&  \iint_Q g w + \int_{\Omega} p(\cdot, 0) w(\cdot, 0) \\ \nonumber 
%&-& \int_{\Omega} p(\cdot, T) w(\cdot, T),\quad \forall\, w\in L^2(0, T; V), 
%\end{eqnarray}
%where the differentials in the integrals are omitted for the sake of brevity. 

(II) We then consider the overlapping decomposition in time domain, i.e., $T_1^{r} > T_2^ l$. 
Denote by $\widetilde{y}= y_1-y_2$ and $\widetilde{p}= p_1-p_2$ on $Q^*= Q_1 \bigcap Q_2$, then 
\eq \label{opt_overlapping1}
\left\{\begin{array}{rl}
-\partial_t \widetilde{y} + \Delta \widetilde{y} - \widetilde{p}/\gamma=&0\ \tn{in} Q^*,\quad \\
\widetilde{y}=&0\ \tn{on} \Sigma,\quad
\widetilde{y}(\cdot, T_2^ l)= 0 \tn{in}\ \Omega,\\
\partial_t \widetilde{p}+ \Delta \widetilde{p}+ \widetilde{y} =& 0\ \tn{in} Q^*,\quad \\
\widetilde{p} =&0\ \tn{on} \Sigma, \quad
\widetilde{p} (\cdot,T_1^{r})=0 \tn{in}\ \Omega.
\end{array}\right.
\ee
We first show $\widetilde{y}, \widetilde{p} \in W(T_2^ l, T_1^r; Q^*)$ is identically zero. 
%We use the variational formulations of $\widetilde{y}$ and $\widetilde{p}$.  
In the equation of $\widetilde{y}$, taking $\widetilde{p}$ as the test function, we have 
\begin{equation}
-\iint_{Q^*}\partial_t \widetilde{y}\, \widetilde{p}+ \iint_{Q^*}\Delta \widetilde{y}\, \widetilde{p} - \frac{1}{\gamma}\iint_{Q^*} \widetilde{p}^2 = 0.
\label{eq:opt_overlapping1_1}
\end{equation}
After integration by parts, it can be written as 
\begin{equation}
\iint_{Q^*} \widetilde{y}\, \partial_t\widetilde{p} - 
\int_{\Omega}\left[\widetilde{y}(T_1^{r})\widetilde{P}(T_1^{r}) - \widetilde{y}(T_2^ l)\widetilde{P}(T_2^ l)\right] 
- \iint_{Q^*}\nabla \widetilde{y}\,\nabla \widetilde{p} 
+ \int_{\Sigma} \frac{\partial \widetilde{y}}{\partial n} \widetilde{p} 
- \frac{1}{\gamma}\iint_{Q^*} \widetilde{p}^2 = 0.
\label{eq:opt_overlapping1_2}
\end{equation}
Analogously, we take $\widetilde{y}$ as the test function in the equation for $\widetilde{p}$ and get 
\begin{equation}
\iint_{Q^*}\partial_t \widetilde{p}\, \widetilde{y}+ \iint_{Q^*}\Delta \widetilde{p}\, \widetilde{y} + \iint_{Q^*} \widetilde{y}^2 = 0.
\label{eq:opt_overlapping1_3}
\end{equation}
Integration by parts on the second term leads to  
\begin{equation}
\iint_{Q^*} \partial_t\widetilde{p}\, \widetilde{y} - \iint_{Q^*}\nabla \widetilde{p}\, \nabla \widetilde{y} + \int_{\Sigma} \frac{\partial \widetilde{p}}{\partial n} \widetilde{y} + \iint_{Q^*} \widetilde{y}^2  = 0.
\label{eq:opt_overlapping1_4}
\end{equation}
Subtracting \eqref{eq:opt_overlapping1_2} from \eqref{eq:opt_overlapping1_4} and considering the boundary and initial conditions, we have 
\begin{equation}
\iint_{Q^*} \widetilde{y}^2 + \frac{1}{\gamma}\iint_{Q^*} \widetilde{p}^2 = 0. 
\end{equation}
It indicates that $\widetilde{y}$ and $\widetilde{p}$ are identically zero on $Q^*$, therefore, $y_1 = y_2$ and $p_1= p_2$ on $Q^*$. 

Therefore, the overlapping case can be regarded as a decomposition of two non-overlapping subdomains $\Omega\times(0, T_1^r)$ and $\Omega\times(T_1^r, T)$. 
Similar to the case (I), it is easy to show that the hybrid formulation \eqref{tdd_2s_1} and \eqref{tdd_2s_2} is equivalent to the optimality system \eqref{opt1A}. 
This completes the proof.
%Let $\chi_1$ and $\chi_2$ be a sufficiently smooth partition of unity subordinate to the cover $Q_1$ and $Q_2$. 
%It is seen that $y = \chi_1 y_1 + \chi_2 y_2$ and $p = \chi_1 p_1 + \chi_2 p_2$ satisfy \eqref{opt1A}. 
\end{proof}
%The above conclusion can be generalized to many subdomains,
%which consequently leads to our proposed domain decomposition algorithms in the next section.
The above hybrid formulation \eqref{tdd_2s_1}-\eqref{tdd_2s_2} can be easily extended to the multi-domain case, which motivates our iterative domain decomposition algorithms {introduced} in the next section.  
%Two requirements: 1. the restriction $y_i(x)$ of the true solution $y(x)$ to each subdomain $\Omega_i$ must solve the hybrid system; 2. the hybrid formulation must be well posed as a couple system. $y_i(x)$ must exist and be unique, and it must depend continuously on the data. The solution $y(x)$ can be expressed in terms of the local solutions as 
%$y(x)=\chi_1(x) y_1(x)+\chi_2(x) y_2(x)$. 
%
%Matching conditions are equations satisfied by the true solution $y(x)$ restricted to the interfaces or regions of overlap between adjacent subdomains. 
%For the optimality system, due to the requirement of continuity in time of the local solutions $y_i$ and $p_i$ across adjacent time subdomains. 
%We have $y_i= y_j$ and $p_i = p_j$ on $\partial \Omega_i \cap \Omega_j$ or $\partial \Omega_i^* \cap \Omega_j^*$. 

%=============================================================
\section{One-level Time Domain Decomposition Algorithms with Convergence Analysis}\label{sec: itm}
%=============================================================
To take advantage of massively parallel computers, 
we need to consider many subdomains so that the computational task for solving the original optimality system can be allocated to multiple processors and be solved in a parallel manner.
In this paper, we consider a strip decomposition in the time domain. 
%Since the forward-backward PDE system produces a huge system when $T$ is large, which makes numerical simulations computationally expensive. 
In nonoverlapping algorithms, we partition the time domain $[0, T]$ into $K$ non-overlapping subdomains of equal length. 
For overlapping algorithms, we extend each nonoverlapping subdomain to its neighboring subdomains with an overlap length $\delta>0$. 
Denote by $Q_i = \Omega\times (T_i^ l, T_i^{r})$ the $i$-th subdomain. 
For the first and last subdomain, we will enforce the maximum length of {an} interval by setting $T_1^ l=0$ and $T_K^{r}=T$. 
A typical scenario of 3 overlapping subdomains is illustrated in Figure \ref{fig: domain}.
\begin{figure}[h!]
\centering
\includegraphics[width=0.618\textwidth]{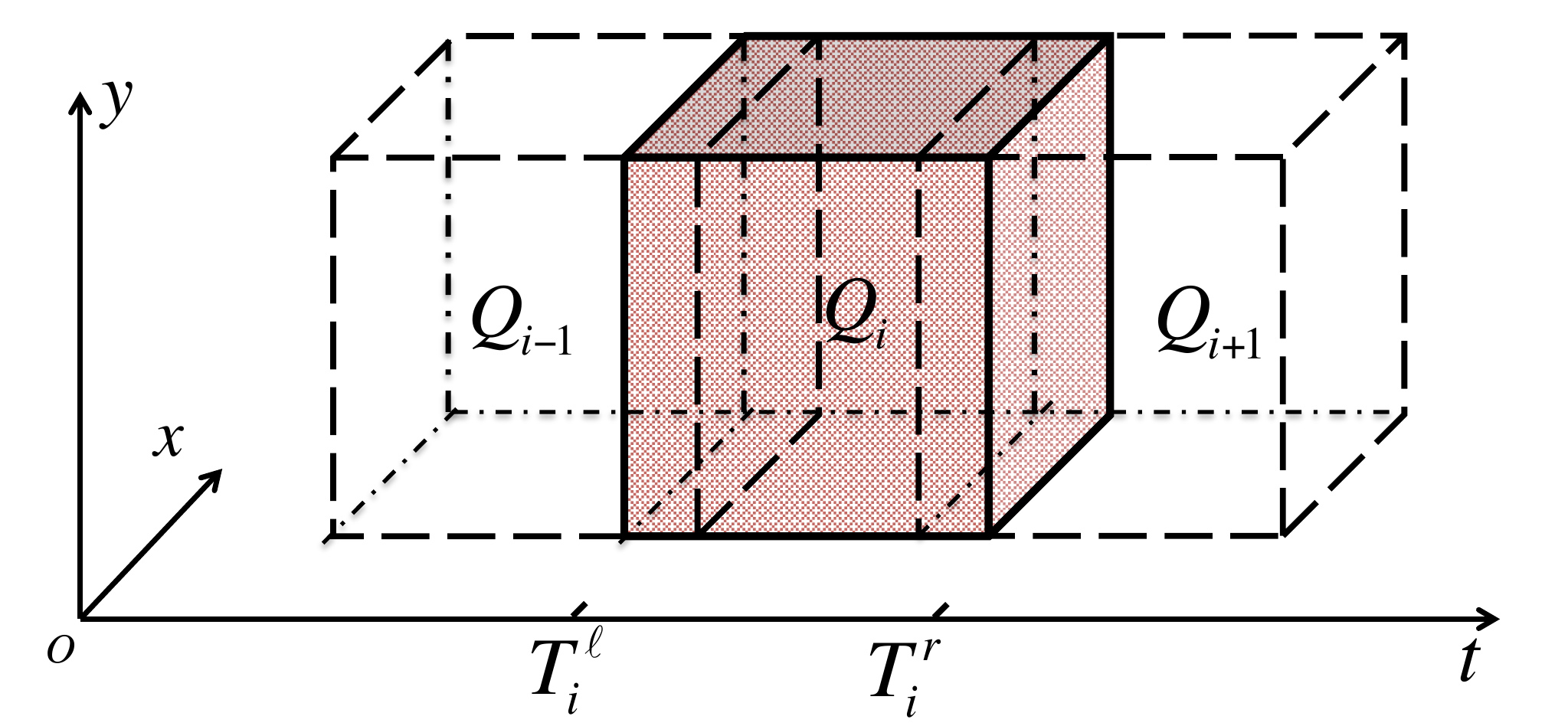}
\caption{A schematic diagram of strip domain decomposition in time.}
\label{fig: domain}
\end{figure}

%If $y_i^{(k)}$ and $p_i^{(k)}$ denote the $k$-th iteration on subdomain $Q_i= \Omega \times R_i$, for $i= 1, \ldots, N_p$, they can be updated by solving \eqref{tdd_2s_1}-\eqref{tdd_2s_2} with boundary conditions $p_1(\cdot, T_1^{r})$ and $y_2(\cdot, T_2^ l)$ that can be approximated by the current iterate. 
%Hence, we propose {\em multidomain Schwarz iterative algorithms}, which can be categorized into two classes based on the ways of updating the interface conditions. 
For simplicity of exposition, we will present our algorithms and convergence analysis for the case with only two subdomains (i.e., $K=2$), but the generalization of our results to many subdomains are straightforward and the convergence {for multiple subdomains} will also be demonstrated by our numerical results.
Suppose the state $y_i^{k-1}$ and adjoint state $p_i^{k-1}$ on $Q_i (i=1,2)$  are given at the $(k-1)$-th iteration, 
then the new states $y_i^{k}$ and $p_i^{k}$ satisfy:
\eq \label{tdd_msch_1}
\left\{\begin{array}{rl}
-\partial_t y_1^{k} + \Delta y_1^{k} -p_1^{k}/\gamma=&f\ \tn{in} Q_1, \quad 
y_1^{k}(\cdot,T_1^l=0)=y_0 \tn{in}\ \Omega,\\
\partial_t p_1^{k}+ \Delta p_1^{k}+ y_1^{k}=& g\ \tn{in} Q_1,\quad
p_1^{k}(\cdot,T_1^r)= p_{2}^{k-1}(T_1^r)\ \tn{in}\ \Omega,
\end{array}\right.\\ \newline \\
\left\{\begin{array}{rl}
-\partial_t y_2^{k} + \Delta y_2^{k} -p_2^{k}/\gamma=&f\ \tn{in} Q_2, \quad 
y_2^{k}(\cdot,T_2^l)=y_1^{\star}(\cdot,T_2^l) \tn{in}\ \Omega,\\
\partial_t p_2^{k}+ \Delta p_2^{k}+ y_2^{k}=& g\ \tn{in} Q_2,\quad
p_2^{k}(\cdot,T_2^r=T)=0\ \tn{in}\ \Omega,
\end{array}\right.
\ee
where the original homogeneous boundary conditions $y_1=0, p_1=0, y_2=0$, and $p_2=0$ hold for all $k$.

When $\star=(k-1)$ in \eqref{tdd_msch_1}, we obtain an {\em additive} Schwarz algorithm. 
It solves the optimality system in a highly parallel manner. 
Since interface conditions come from the previous iteration, the new approximations $y_{i}^{k}$ and $p_{i}^{k}$ on each subdomain $Q_i$ can be computed concurrently. 

When $\star= k$ in \eqref{tdd_msch_1}, we get a {\em multiplicative} Schwarz algorithm. 
It solves the optimality system by sequentially updating the approximations on subdomains in a prescribed order. 
%It is sequential in nature since the solution of one subdomain problem prior to another. 
In fact, the latest interface value $y_{i-1}^{k}(\cdot, T_{i-1}^ r)$ from the preceding subdomain $Q_{i-1}$ is
immediately used in updating $y_{i}^{k}$ and $p_{i}^{k}$ on the subdomain $Q_{i}$. 
Therefore, the multiplicative algorithm runs sequentially in nature. 
As a standard strategy, its parallelizability can be greatly improved by grouping the subdomains into different colors. 
The subdomains with the same color do not intersect with each other and, thus, can be solved simultaneously. 
Indeed, since we consider the decomposition in the 1D time direction, two colors are sufficient to group all the strips. 

We will study four different domain decomposition algorithms: 
(1) additive Schwarz iterations with no overlap (ASN);
(2) multiplicative Schwarz iterations with no overlap (MSN); 
(3) additive Schwarz iterations with overlap (ASO);
(4) multiplicative Schwarz iterations with overlap (MSO).
With comprehensive numerical tests, we will compare the convergence performance of each of them (1-level and 2-level) as stand-alone iterative solver and preconditioner of GMRES Krylov subspace solver, respectively.

%=============================================================
%\section{Convergence analysis with two subdomains}\label{sec: itm}
%=============================================================
Next, we analyze the convergence of the non-overlapping algorithms, ASN and MSN, respectively. 
% which is more frequently used in practice due to its better parallelism performance than MSN.
In the following, we assume the case with two subdomains.
Denote $T_1=T_2^l=T_1^r\in (0,T)$ and the global time interval $[0,T]$ is then divided into two nonoverlapping sub-intervals $[0,T_1]$ and $[T_1,T]$.
Upon a semi-discretization of the Laplacian operator with a second-order central difference scheme in space,
the ASN algorithm based on (\ref{tdd_msch_1}) iterates concurrently according to the following two coupled ODEs
\eq \label{tdd_msch_ASN1}
\left\{\begin{array}{rl}
-\partial_t y_1^{k} + \Delta_h y_1^{k} -p_1^{k}/\gamma=&f_1\ \tn{in} Q_1, \quad 
y_1^{k}(0)=y_0 \tn{in}\ \Omega,\\
\partial_t p_1^{k}+ \Delta_h p_1^{k}+ y_1^{k}=& g_1\ \tn{in} Q_1,\quad
p_1^{k}(T_1)= p_{2}^{k-1}(T_1)\ \tn{in}\ \Omega,
\end{array}\right.
\ee
\eq \label{tdd_msch_ASN2}
\left\{\begin{array}{rl}
-\partial_t y_2^{k} + \Delta_h y_2^{k} -p_2^{k}/\gamma=&f_2\ \tn{in} Q_2, \quad 
y_2^{k}(T_1)=y_1^{k-1}(T_1) \tn{in}\ \Omega,\\
\partial_t p_2^{k}+ \Delta_h p_2^{k}+ y_2^{k}=& g_2\ \tn{in} Q_2,\quad
p_2^{k}(T)=0\ \tn{in}\ \Omega,
\end{array}\right.
\ee
where  {$y_1^{k}(t), y_2^{k}(t), p_1^{k}(t), p_2^{k}(t), f_1(t), f_2(t), g_1(t), g_2(t)$ are the corresponding spatially discretized vector functions} and 
$\Delta_h$ is a discrete Laplacian operator associated with the Dirichlet boundary conditions. 
 {It is worthwhile to point out that, although we use the discrete Laplacian with a 5-point stencil central difference in the full discrete scheme (\ref{FDEq1}-\ref{FDEq22}), our following analysis also applies to {other discretization} of the Laplacian operator}. 
The initial and ending vectors $y_1^{k-1}(T_1), p_{2}^{k-1}(T_1)$ of $(k-1)$-th iteration are assumed to be given or already obtained.
In the following theorem we proved the convergence of the above semi-discretized ASN algorithm, by treating time in a continuous manner.
\begin{theorem}\label{thmASN}
 The ASN algorithm given by the iterations (\ref{tdd_msch_ASN1}-\ref{tdd_msch_ASN2}) is convergent.
\end{theorem}
\begin{proof}
According to Theorem \ref{thm1}, subtracting (\ref{tdd_msch_ASN1}-\ref{tdd_msch_ASN2}) from the corresponding semi-discretized version of (\ref{tdd_2s_1}-\ref{tdd_2s_2})
leads to the following coupled system of the approximation error functions
\eq \label{tdd_msch_ASN1err}
\left\{\begin{array}{rl}
-\partial_t e_1^{k} + \Delta_h e_1^{k} -w_1^{k}/\gamma=&0\ \tn{in} Q_1, \quad 
e_1^{k}(0)=0 \tn{in}\ \Omega,\\
\partial_t w_1^{k}+ \Delta_h w_1^{k}+ e_1^{k}=& 0\ \tn{in} Q_1,\quad
w_1^{k}(T_1)= w_{2}^{k-1}(T_1)\ \tn{in}\ \Omega,
\end{array}\right.
\ee
\eq \label{tdd_msch_ASN2err}
\left\{\begin{array}{rl}
-\partial_t e_2^{k} + \Delta_h e_2^{k} -w_2^{k}/\gamma=&0\ \tn{in} Q_2, \quad 
e_2^{k}(T_1)=e_1^{k-1}(T_1) \tn{in}\ \Omega,\\
\partial_t w_2^{k}+ \Delta_h w_2^{k}+ e_2^{k}=& 0\ \tn{in} Q_2,\quad
w_2^{k}(T)=0\ \tn{in}\ \Omega,
\end{array}\right.
\ee
where $e_1^k(t)=y_1^k(t)-y_1(t)$, $e_2^k(t)=y_2^k(t)-y_2(t)$, $w_1^k(t)=p_1^k(t)-p_1(t)$, $w_2^k(t)=p_2^k(t)-p_2(t)$
are the error vector functions of the state $y$ and adjoint state $p$ on each sub-interval.  
If we can show that both $e_1^{k}(T_1)$ and $w_{2}^{k}(T_1)$ converge to zero as $k$ goes to infinity,
then the uniqueness of the solution to both (\ref{tdd_msch_ASN1err}) and (\ref{tdd_msch_ASN2err}) obviously
implies all error vector functions over each sub-interval also converge to zero as $k$ goes to infinity.

In (\ref{tdd_msch_ASN1err}), it follows from multiplying from the left side of the first equation with {$(e_1^{k})^\intercal$}, 
the transpose of $(e_1^{k})$, and the second one with $(w_1^{k})^\intercal $, respectively,
\eq \label{tdd_msch_ASN1err1}
\left\{\begin{array}{rl}
-\frac{1}{2} \partial_t [(e_1^{k})^\intercal e_1^{k}]+ (e_1^{k})^\intercal \Delta_h e_1^{k} -(e_1^{k})^\intercal w_1^{k}/\gamma=&0\ \tn{in} Q_1, \quad 
e_1^{k}(0)=0 \tn{in}\ \Omega,\\
\frac{1}{2} \partial_t [(w_1^{k})^\intercal w_1^{k}]+ (w_1^{k})^\intercal \Delta_h w_1^{k}+ (w_1^{k})^\intercal e_1^{k}=& 0\ \tn{in} Q_1,\quad
w_1^{k}(T_1)= w_{2}^{k-1}(T_1)\ \tn{in}\ \Omega.
\end{array}\right.
\ee
Since $(e_1^{k})^\intercal w_1^{k}= (w_1^{k})^\intercal e_1^{k}$, the addition of the first equation multiplied with $\gamma$  and the second one leads to
\eq \label{tdd_msch_ASN1err2}
\left\{\begin{array}{rl}
-\frac{1}{2} \partial_t \gamma [(e_1^{k})^\intercal e_1^{k}]+ \frac{1}{2} \partial_t [(w_1^{k})^\intercal w_1^{k}]+
\gamma (e_1^{k})^\intercal \Delta_h e_1^{k} + (w_1^{k})^\intercal \Delta_h w_1^{k} =&0\ \tn{in} Q_1, \quad \\
e_1^{k}(0)=0,\quad  w_1^{k}(T_1)= w_{2}^{k-1}(T_1)\ \tn{in}\ \Omega.
\end{array}\right.
\ee
Applying the Fundamental Theorems of Calculus, a straightforward integration of the above equation from $0$ to $T_1$ gives
\eq \label{tdd_msch_ASN1err3} 
\left\{\begin{array}{rl}
\frac{1}{2} \gamma [(e_1^{k}(0))^\intercal e_1^{k}(0)-(e_1^{k}(T_1))^\intercal e_1^{k}(T_1)]+ \frac{1}{2}[(w_1^{k}(T_1))^\intercal w_1^{k}(T_1)-(w_1^{k}(0))^\intercal w_1^{k}(0)]+\\
\int_0^{T_1} [\gamma (e_1^{k})^\intercal \Delta_h e_1^{k}]dt + \int_0^{T_1} [(w_1^{k})^\intercal \Delta_h w_1^{k}] dt=&0\ \tn{in} Q_1, \quad \\
e_1^{k}(0)=0,\quad  w_1^{k}(T_1)= w_{2}^{k-1}(T_1)\ \tn{in}\ \Omega. 
\end{array}\right.
\ee
It then follows by enforcing the given end-point conditions and rearranging the terms (using 2-norm $\|\cdot \|$)
\eq \label{tdd_msch_ASN1err4}  
 \gamma \|e_1^{k}(T_1)\|^2= \|w_2^{k-1}(T_1)\|^2-\|w_1^{k}(0)\|^2-
\int_0^{T_1} 2\gamma \|\nabla_h e_1^{k}\|^2 dt - \int_0^{T_1} 2\|\nabla_h w_1^{k}\|^2 dt,   
\ee
where we have used the fact (based on the discrete version of integration by parts \cite{Jovanovic2014})
$$z^\intercal (-\Delta_h) z=\langle z,(-\Delta_h) z\rangle=\langle\nabla_h z,\nabla_h z\rangle=\|\nabla_h z\|^2.$$

Similarly in \eqref{tdd_msch_ASN2err}, with an integration from $T_1$ to $T$ , we will get
 \eq \label{tdd_msch_ASN2err3} 
\left\{\begin{array}{rl}
\frac{1}{2} \gamma [(e_2^{k}(T_1))^\intercal e_2^{k}(T_1)-(e_2^{k}(T))^\intercal e_2^{k}(T)]+ \frac{1}{2}[(w_2^{k}(T))^\intercal w_2^{k}(T)-(w_2^{k}(T_1))^\intercal w_2^{k}(T_1)]+\\
\int_{T_1}^T [\gamma (e_2^{k})^\intercal \Delta_h e_2^{k}]dt + \int_{T_1}^T [(w_2^{k})^\intercal \Delta_h w_2^{k}] dt=&0\ \tn{in} Q_2, \quad \\
e_2^{k}(T_1)=e_1^{k-1}(T_1),\quad  w_2^{k}(T)= 0\ \tn{in}\ \Omega,
\end{array}\right.
\ee
from which we obtain
\eq \label{tdd_msch_ASN2err4}   
\|w_2^{k}(T_1)\|^2= \gamma \|e_1^{k-1}(T_1)\|^2-\gamma \|e_2^{k}(T)\|^2-
\int_{T_1}^T 2\gamma \|\nabla_h e_2^{k}\|^2 dt - \int_{T_1}^T 2\|\nabla_h w_2^{k}\|^2 dt.
\ee
Add (\ref{tdd_msch_ASN1err4}) and (\ref{tdd_msch_ASN2err4}) together to get
\eq \label{tdd_msch_ASNerr5}   
\gamma \|e_1^{k}(T_1)\|^2+\|w_2^{k}(T_1)\|^2&=\gamma \|e_1^{k-1}(T_1)\|^2+\|w_2^{k-1}(T_1)\|^2
-\|w_1^{k}(0)\|^2-\gamma \|e_2^{k}(T)\|^2\\
&-\int_0^{T_1} 2\gamma \|\nabla_h e_1^{k}\|^2 dt - \int_0^{T_1} 2\|\nabla_h w_1^{k}\|^2 dt 
-\int_{T_1}^T 2\gamma \|\nabla_h e_2^{k}\|^2 dt - \int_{T_1}^T 2\|\nabla_h w_2^{k}\|^2 dt\\
&\le 
\gamma \|e_1^{k-1}(T_1)\|^2+\|w_2^{k-1}(T_1)\|^2
-\|w_1^{k}(0)\|^2-\gamma \|e_2^{k}(T)\|^2\\
&-\int_0^{T_1} 2\gamma c \|e_1^{k}\|^2 dt - \int_0^{T_1} 2 c \| w_1^{k}\|^2 dt 
-\int_{T_1}^T 2\gamma c\|e_2^{k}\|^2 dt - \int_{T_1}^T 2 c\|w_2^{k}\|^2 dt\\
&\le \gamma \|e_1^{k-1}(T_1)\|^2+\|w_2^{k-1}(T_1)\|^2 \,,
\ee
where in the second last step we used the following inequalities (based on the discrete version of Poincar\'e inequality)
\[
 \|e_1^{k}\|^2\le (1/c) \|\nabla_h e_1^{k}\|^2, \|e_2^{k}\|^2\le (1/c) \|\nabla_h e_2^{k}\|^2,
 \|w_1^{k}\|^2\le (1/c) \|\nabla_h w_1^{k}\|^2,  \|w_2^{k}\|^2\le (1/c) \|\nabla_h w_2^{k}\|^2, 
\]
with a positive constant $c$  independent of $h$, $T$ and $T_1$.

From (\ref{tdd_msch_ASNerr5}), we can easily see that the sequence $\{\gamma \|e_1^{k}(T_1)\|^2+\|w_2^{k}(T_1)\|^2\}_{k=0}^{\infty}$ is decreasing {monotonically and bounded below.}
It hence implies (note $\gamma>0$), { there exists a constant $\lambda\ge 0$, such that }
\[
 \lim_{k\to\infty } \gamma \|e_1^{k}(T_1)\|^2+\|w_2^{k}(T_1)\|^2= {\lambda}.
\]
{Then, by letting $k\to\infty$ in the first equality of \eqref{tdd_msch_ASNerr5}, we conclude that 
\[ \lim_{k\to\infty } \|w_1^{k}(0)\|^2+\gamma \|e_2^{k}(T)\|^2 
+\int_0^{T_1} 2(\gamma \|\nabla_h e_1^{k}\|^2 + \|\nabla_h w_1^{k}\|^2)  dt
+\int_{T_1}^{T} 2(\gamma \|\nabla_h e_2^{k}\|^2  +  \|\nabla_h w_2^{k}\|^2) dt=0,
\]
which further implies
\[\lim_{k\to\infty }\|\nabla_h e_1^k(t)\|= \lim_{k\to\infty }\|\nabla_h w_1^k(t)\|=0\] on $(0, T_1)$ and 
\[\lim_{k\to\infty }\|\nabla_h e_2^k(t)\|= \lim_{k\to\infty }\|\nabla_h w_2^k(t)\|= 0\] over $(T_1, T)$. 
Since $e_1$, $e_2$, $w_1$, $w_2$ at any time $t$ all belong to $H^1_0(\Omega)$, 
we have \[\lim_{k\to\infty } e_1^k (t)=\lim_{k\to\infty } w_1^k (t)=0\] and 
\[\lim_{k\to\infty } e_2^k (t)=\lim_{k\to\infty } w_2^k (t)=0\] on the time intervals $(0, T_1)$ and $(T_1, T)$, respectively.}
Therefore, for any initial guess, both $e_1^{k}(T_1)$ and $w_2^{k}(T_1)$ converge to zero as $k\to \infty$.
This completes the proof of the convergence of our ASN algorithm.
\end{proof}

Repeating  the above arguments, we can obtain the convergence results for the following semi-discretized MSN algorithm: 
\eq \label{tdd_msch_MSN1}
\left\{\begin{array}{rl}
-\partial_t y_1^{k} + \Delta_h y_1^{k} -p_1^{k}/\gamma=&f_1\ \tn{in} Q_1, \quad 
y_1^{k}(0)=y_0 \tn{in}\ \Omega,\\
\partial_t p_1^{k}+ \Delta_h p_1^{k}+ y_1^{k}=& g_1\ \tn{in} Q_1,\quad
p_1^{k}(T_1)= p_{2}^{k-1}(T_1)\ \tn{in}\ \Omega,
\end{array}\right.
\ee
\eq \label{tdd_msch_MSN2}
\left\{\begin{array}{rl}
-\partial_t y_2^{k} + \Delta_h y_2^{k} -p_2^{k}/\gamma=&f_2\ \tn{in} Q_2, \quad 
y_2^{k}(T_1)=y_1^{k}(T_1) \tn{in}\ \Omega,\\
\partial_t p_2^{k}+ \Delta_h p_2^{k}+ y_2^{k}=& g_2\ \tn{in} Q_2,\quad
p_2^{k}(T)=0\ \tn{in}\ \Omega.
\end{array}\right.
\ee
To illustrate the convergence difference between ASN and MSN algorithm, we briefly sketch the proof below. 
\begin{theorem} \label{thmMSN}
 The MSN algorithm given by the iterations (\ref{tdd_msch_MSN1}-\ref{tdd_msch_MSN2}) is convergent.
\end{theorem}
\begin{proof}
 According to Theorem \ref{thm1}, subtracting (\ref{tdd_msch_MSN1}-\ref{tdd_msch_MSN2}) from the corresponding semi-discretized version of (\ref{tdd_2s_1}-\ref{tdd_2s_2})
leads to the following coupled system of the approximation error functions
\eq \label{tdd_msch_MSN1err}
\left\{\begin{array}{rl}
-\partial_t e_1^{k} + \Delta_h e_1^{k} -w_1^{k}/\gamma=&0\ \tn{in} Q_1, \quad 
e_1^{k}(0)=0 \tn{in}\ \Omega,\\
\partial_t w_1^{k}+ \Delta_h w_1^{k}+ e_1^{k}=& 0\ \tn{in} Q_1,\quad
w_1^{k}(T_1)= w_{2}^{k-1}(T_1)\ \tn{in}\ \Omega,
\end{array}\right.
\ee
\eq \label{tdd_msch_MSN2err}
\left\{\begin{array}{rl}
-\partial_t e_2^{k} + \Delta_h e_2^{k} -w_2^{k}/\gamma=&0\ \tn{in} Q_2, \quad 
e_2^{k}(T_1)=e_1^{k}(T_1) \tn{in}\ \Omega,\\
\partial_t w_2^{k}+ \Delta_h w_2^{k}+ e_2^{k}=& 0\ \tn{in} Q_2,\quad
w_2^{k}(T)=0\ \tn{in}\ \Omega.
\end{array}\right.
\ee
In (\ref{tdd_msch_MSN1err}), it follows from the same procedure that
\eq \label{tdd_msch_MSN1err4}  
 \gamma \|e_1^{k}(T_1)\|^2= \|w_2^{k-1}(T_1)\|^2-\|w_1^{k}(0)\|^2-
\int_0^{T_1} 2\gamma \|\nabla_h e_1^{k}\|^2 dt - \int_0^{T_1} 2\|\nabla_h w_1^{k}\|^2 dt.   
\ee 
Similarly,  in (\ref{tdd_msch_MSN2err}), we can obtain the slightly different
\eq \label{tdd_msch_MSN2err4}   
\|w_2^{k}(T_1)\|^2={ \gamma \|e_1^{k}(T_1)\|^2}-\gamma \|e_2^{k}(T)\|^2-
\int_{T_1}^T 2\gamma \|\nabla_h e_2^{k}\|^2 dt - \int_{T_1}^T 2\|\nabla_h w_2^{k}\|^2 dt.
\ee
Adding (\ref{tdd_msch_MSN1err4}) and (\ref{tdd_msch_MSN2err4}) together, we have 
\eq \label{tdd_msch_MSNerr5}   
\|w_2^{k}(T_1)\|^2&=\|w_2^{k-1}(T_1)\|^2
-\|w_1^{k}(0)\|^2-\gamma \|e_2^{k}(T)\|^2\\
&-\int_0^{T_1} 2\gamma \|\nabla_h e_1^{k}\|^2 dt - \int_0^{T_1} 2\|\nabla_h w_1^{k}\|^2 dt 
-\int_{T_1}^T 2\gamma \|\nabla_h e_2^{k}\|^2 dt - \int_{T_1}^T 2\|\nabla_h w_2^{k}\|^2 dt\\
&\le 
\|w_2^{k-1}(T_1)\|^2
-\|w_1^{k}(0)\|^2-\gamma \|e_2^{k}(T)\|^2\\
&-\int_0^{T_1} 2\gamma c \|e_1^{k}\|^2 dt - \int_0^{T_1} 2 c \| w_1^{k}\|^2 dt 
-\int_{T_1}^T 2\gamma c\|e_2^{k}\|^2 dt - \int_{T_1}^T 2 c\|w_2^{k}\|^2 dt\\
&\le \|w_2^{k-1}(T_1)\|^2,
\ee
from which we can easily see that the nonnegative sequence $\{\|w_2^{k}(T_1)\|^2\}_{k=0}^{\infty}$ is { monotonically } decreasing {and bounded below, we conclude the limit exists as $k\rightarrow \infty$}.
{Following the same argument as that in Theorem \ref{thmASN}, we have}, 
for any initial guess, both $e_1^{k}(T_1)$ and $w_2^{k}(T_1)$ converge to zero as $k\to \infty$.
This completes the proof of the convergence of our MSN algorithm.
\end{proof}
%%
%%\red{Hence there holds
%%\[
%% \lim_{k\to\infty }  \|w_2^{k}(T_1)\|^2=\lambda
%%\]
%%for some constant $\lambda$. 
%%It then follows from (\ref{tdd_msch_MSNerr5}) that 
%%\[ \lim_{k\to\infty } \|w_1^{k}(0)\|^2+\gamma \|e_2^{k}(T)\|^2 
%%+\int_0^{T_1} (2\gamma \|\nabla_h e_1^{k}\|^2 + 2\|\nabla_h w_1^{k}\|^2  
%%+ 2\gamma \|\nabla_h e_2^{k}\|^2  +  2\|\nabla_h w_2^{k}\|^2) dt=0,
%%\]
%%which clearly implies 
%%\eq \label{limitw10}
%% w_1(0):=\lim_{k\to\infty } w_1^{k}(0)=0.
%%\ee
%%By taking $k\to\infty$ in the error system (\ref{tdd_msch_MSN1err}), we get the limiting system (dropping the supscript $k$)
%%\eq \label{tdd_msch_MSN1err_limit}
%%\left\{\begin{array}{rl}
%%-\partial_t e_1 + \Delta_h e_1 -w_1/\gamma=&0\ \tn{in} Q_1, \quad 
%%e_1(0)=0 \tn{in}\ \Omega,\\
%%\partial_t w_1+ \Delta_h w_1+ e_1=& 0\ \tn{in} Q_1,\quad
%%w_1(T_1)= \lambda\ \tn{in}\ \Omega,
%%\end{array}\right.
%%\ee  
%%with an extra initial condition  $w_1(0)=0$ from (\ref{limitw10}).
%%When we treat the system as an initial value problem with initial conditions $e_1(0)=0$ and $w_1(0)=0$,
%%obviously $e_1=w_1\equiv 0$ is a solution to the system (\ref{tdd_msch_MSN1err_limit}).
%%Nevertheless, the solution to (\ref{tdd_msch_MSN1err_limit}) is unique, which implies $\lambda=0$. 
%%Otherwise, as a boundary value problem with $\lambda>0$, it would give a nonzero solution.
%%
%%It then follows from (\ref{tdd_msch_MSN1err4}) that
%%\[
%% \lim_{k\to\infty }   \|e_1^{k}(T_1)\|^2=\lambda/\gamma=0.
%%\]
%%}

The above conclusions are {counter-intuitive and surprising} in the view of the well-recognized fact that
a standard ASN algorithm without overlap for elliptic boundary value problems is not convergent as a stand-alone iterative solver \cite{Efstathiou_2003,Gander2008}.
Nonetheless, the above convergence result does not provide any explicit estimate of the convergence rate of the ASN algorithm, 
which will be carried out further in the following by making use of more detailed information contained in (\ref{tdd_msch_ASNerr5}).
Let $y^k(t)$ be the approximation of $y(t)$ at {the} $k$-th iteration by gluing $y_1^k$ and $y_2^k$ together (taking the value of $y_1^k$ at $T_1$). 
Similarly, let $p^k(t)$ be the approximation of $p(t)$ at $k$-th iteration by gluing $p_1^k$ and $p_2^k$ together (taking the value of $p_2^k$ at  $T_1$).  
Define the global error vector functions $e^k(t)=y^k(t)-y(t)$ and $w^k(t)=p^k(t)-p(t)$. 
Then there holds (in the sense of Lebesgue integral)
\[
 \int_0^{T_1} 2\gamma c \|e_1^{k}\|^2 dt+\int_{T_1}^T 2\gamma c\|e_2^{k}\|^2 dt =2c \int_0^T \gamma \|e^k\|^2 dt
\]
and
\[
 \int_0^{T_1} 2 c \| w_1^{k}\|^2 dt +\int_{T_1}^T 2 c\|w_2^{k}\|^2 dt=2c \int_0^T \|w^k\|^2 dt.
\]
Let $z^k(t):=\gamma \|e^{k}(t)\|^2+\|w^{k}(t)\|^2$. Then by definition we have  $z^k(T_1)=\gamma \|e_1^{k}(T_1)\|^2+\|w_2^{k}(T_1)\|^2$.
By using this new vector function $z^k(t)$ and dropping the non-positive terms $(-\|w_1^{k}(0)\|^2-\gamma \|e_2^{k}(T)\|^2)$, it follows from (\ref{tdd_msch_ASNerr5}) that
\eq \label{tdd_msch_ASNerr6}   
z^k(T_1)\le z^{k-1}(T_1)-2c\int_0^T z^k(t) dt
\ee
holds for any fixed $T_1\in (0,T)$. 
% {Notice that by its definition, $z^k(t)$ depends on the choice of $T_1$. Hence, we can not directly perform 
%a simple integration on both sides of (\ref{tdd_msch_ASNerr6}) with respect to $T_1$, to eliminate the dependence on $T_1$.}  {(shall we remove this sentence?)}
Considering that $z^k(0)=\|w_1^k(0)\|^2$ and $z^k(T)=\gamma \|e_2^k(T)\|^2$ are the approximation errors at both end points,
we assume there holds $z^k(t)\ge \b_k z^k(T_1)$ for some $\b_k>0$, i.e., the global minimal approximation error is roughly bounded below by the local error at $t=T_1$.
With this assumption and in view of (\ref{tdd_msch_ASNerr6}), we can arrive
\eq \label{tdd_msch_ASNerr7}   
z^k(T_1)\le z^{k-1}(T_1)-2c\b_k T z^k(T_1),
\ee
which clearly implies
\eq \label{tdd_msch_ASNerr8}   
 {z^k(T_1)}\le \frac{1}{1+2c\b_k T} {z^{k-1}(T_1)},
\ee
where $\b_k$ is expected to be dependent on $k$, $T$, and the problem.
{
In view of (\ref{tdd_msch_ASNerr8}), the convergence rate of the ASN algorithm is mainly determined by the value of $\b_k$.
In particular, numerical simulations show that the convergence rate of 1-level ASN algorithm gets worse as the approximation errors become smaller.
This can be explained by the easy-to-check fact that $\b_k$ indeed decreases to zero as $k$ goes to infinity,
which can also be easily seen from (\ref{tdd_msch_ASNerr6}) by observing $\int_0^T z^k(t) dt\to 0$ as $k$ goes to infinity.
Hence, from {a} theoretical point of view, the 1-level ASN algorithm may converge very {slowly},
especially when the spatial operator has a zero eigenvalue, which was also pointed out in \cite{Gander_2016}.
For our considered spatially uniform elliptic problems, the obtained convergence rates are, however, very satisfactory in numerical {experiments}.
}

{
The convergence analysis of the ASO and MSO algorithms will require some different {proofs},
which are not carried out in current paper. However, we numerically demonstrate the convergence of
both one-level and two-level ASO and MSO algorithms as stand-alone solvers and preconditioners, respectively.
According to our simulations, the ASO and MSO algorithms converge slightly faster than the ASN and MSN algorithms, respectively.
It is well-known that adding a suitable overlap (as in ASO and MSO) in domain decomposition algorithms often leads to better convergence rate that nonoverlapping ones. 
Therefore, it is reasonable for us to believe that the proved convergence of ASN and MSN algorithms 
\textit{nominally} implies (but not rigorously {proves}) the convergence of both ASO and MSO algorithms.
Further convergence analysis of the ASO and MSO algorithms (including the effects of the size of overlap) will be investigated in future. }
% {\color{blue}
% If treating $T_1$ as a new variable, a simple integration of both sides of (\ref{tdd_msch_ASNerr6}) from $T_1=0$ to $T_1=T$ gives
% \eq \label{tdd_msch_ASNerr7}   
% \int_0^T z^k(t) dt\le \frac{1}{(1+2c T)} \int_0^T z^{k-1}(t) dt,
% \ee
% which says that the global solution approximation error $z^k(t)$ converges in $L_2([0,T])$ norm with a uniform convergence rate at least $\frac{1}{\sqrt{1+2c T}}$,
% that is independent of $h$ and even the choice of $T_1$. 
% This conclusion is very surprising, since it implies that we  will get a faster uniform convergence rate for a larger $T$,
% which contradicts with our common intuition. Notice here $c>0$ could be very small but independent of $T$.
% I doubt whether the above arguments are correct or not, since the definition of $z^k(t)$ depends on $T_1$??
% I will go ahead to verify this using numerical tests.}
%\input{numerical}
%=============================================================
\section{Two-level Time Domain Decomposition Algorithms with Many Subdomains\label{sec: two}}
The above introduced one-level domain decomposition algorithms work reasonably well when the number of subdomains $K$ is not large.
However, existing theoretical and numerical results (see also Section \ref{sec: num}) indicate that one-level domain decomposition methods
based only on local subdomain solving are not scalable with respect to the number of subdomains $K$.
Fortunately, this problem can be resolved by using a two-level algorithm in which the current one-level parallel subdomain iterations are 
corrected globally by solving an appropriately selected coarse grid problem with the size of order $O(K)$ \cite{Dolean2015}. 
Such two-level domain decomposition algorithms based on coarse grid corrections can be viewed as a variant of the standard two-level multigrid method,
where the one-level parallel subdomain iterations function as an advanced smoother \cite{Xu_1992,Xu_1998}.
Nevertheless, the coarse grid size in the context of two-level domain decomposition algorithms is usually significantly smaller than
that of a two-level multigrid method based on halving the mesh sizes.  
In general, the successful construction of an effective and efficient coarse grid space is a very dedicated task that often requires many 
 {heuristic numerical investigations by trial and error} \cite{Prudencio_2007}.
In this section, we briefly describe our coarse grid spaces based on the ones discussed in \cite{Martin2012,Gander2014},
which seem to be quite effective according to our following numerical experiments in Section \ref{sec: num}.

We first consider the nonoverlapping cases in MSN and ASN algorithms.
Although it is more convenient to conduct theoretical analysis in the continuous setting,
here we choose to illustrate our algorithms in its discrete formulation 
as this is the most faithful replication of the actual implementation in our codes.
Assume $(N-1)$ is an integer multiple of $K$.
Let the global set of time grid points $\{0=t_0<t_1<\cdots<t_{N-1}<t_N=T\}$ be decomposed into $K$ nonoverlapping subsets of equal size.
Figure \ref{fig:ASN-1} depicts a typical nonoverlapping decomposition scheme with $N=17$ and $K=4$,
where the boundary nodes $t_0=0$ and $t_N=T$ will be absorbed into the right hand side during the computation.
Notice that each subdomain will need to receive two boundary {node} approximations from its neighboring subdomains,
e.g., subdomain 2 ($\{t_5,t_6,t_7,t_8\}$) treats the approximations at nodes $t_4$ and $t_9$ as virtual boundary nodes.
\begin{figure}[!htb] 
    \centering
      \resizebox {\columnwidth} {!} {
 \begin{tikzpicture}[yscale=-1] 
        % 8x8 grid
        \draw (0, 0) grid (17, 0); 
          \draw [color=blue, fill=blue] (0, 0) circle (2pt) node[align=center,   above] {$t_0$};
          \draw [color=blue, fill=blue] (17, 0) circle (2pt) node[align=center,   above] {$t_{17}$};
         \foreach \x in {1,...,16}
        { 
         \draw (\x,0) circle (2pt) node[align=center,   above] {$t_{\x}$};
       }
\draw [decorate,decoration={brace,amplitude=10pt,mirror},xshift=0.4pt,yshift=-0.4pt](1,0) -- (4,0) node[black,midway,yshift=-0.6cm] {subdomain 1};
 \draw [decorate,decoration={brace,amplitude=10pt,mirror},xshift=0.4pt,yshift=-0.4pt](5,0) -- (8,0) node[black,midway,yshift=-0.6cm] {subdomain 2};
\draw [decorate,decoration={brace,amplitude=10pt,mirror},xshift=0.4pt,yshift=-0.4pt](9,0) -- (12,0) node[black,midway,yshift=-0.6cm] {subdomain 3};
 \draw [decorate,decoration={brace,amplitude=10pt,mirror},xshift=0.4pt,yshift=-0.4pt](13,0) -- (16,0) node[black,midway,yshift=-0.6cm] {subdomain 4};
    
    \end{tikzpicture}
    }
    \caption{The nonoverlapping decomposition of all time grid points (empty circles) with $N=17$ and $K=4$.}
      \label{fig:ASN-1}
    \end{figure}
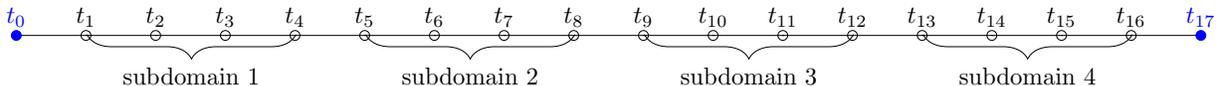
    
Following the approach in \cite{Martin2012,Gander2014}, 
we propose to choose the coarse grid nodes ($\{t_1,t_4,t_5,t_8,t_9,t_{12},t_{13},t_{16}\}$) as described in the following Figure  \ref{fig:ASN-2},
where the adjacent grid nodes connecting two subdomains are selected as coarse grid nodes.
 \begin{figure}[!htb] 
    \centering
      \resizebox {\columnwidth} {!} {
 \begin{tikzpicture}[yscale=-1] 
        % 8x8 grid
        \draw (0, 0) grid (17, 0); 
          \draw [color=blue, fill=blue] (0, 0) circle (2pt) node[align=center,   above] {$t_0$};
          \draw [color=blue, fill=blue] (17, 0) circle (2pt) node[align=center,   above] {$t_{17}$};
         \foreach \x in {1,...,16}
        { 
         \draw (\x,0) circle (2pt) node[align=center,   above] {$t_{\x}$};          
        }
         \foreach \x in {1,4,5,8,9,12,13,16}
        { 
         \filldraw (\x,0) circle (2pt) node[align=center,   above] {$t_{\x}$};                
        }
        
\draw [decorate,decoration={brace,amplitude=10pt,mirror},xshift=0.4pt,yshift=-0.4pt](1,0) -- (4,0) node[black,midway,yshift=-0.6cm] {subdomain 1};
 \draw [decorate,decoration={brace,amplitude=10pt,mirror},xshift=0.4pt,yshift=-0.4pt](5,0) -- (8,0) node[black,midway,yshift=-0.6cm] {subdomain 2};
\draw [decorate,decoration={brace,amplitude=10pt,mirror},xshift=0.4pt,yshift=-0.4pt](9,0) -- (12,0) node[black,midway,yshift=-0.6cm] {subdomain 3};
 \draw [decorate,decoration={brace,amplitude=10pt,mirror},xshift=0.4pt,yshift=-0.4pt](13,0) -- (16,0) node[black,midway,yshift=-0.6cm] {subdomain 4};
    
    \end{tikzpicture}
    }
    \caption{The choice of coarse grid nodes (filled circles) among fine grid points (all circles) with $N=17$ and $K=4$.}
      \label{fig:ASN-2}
    \end{figure}
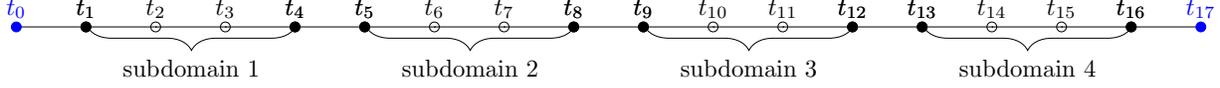
    
After the coarse grid nodes are determined, 
we follow the standard procedure to build the corresponding extension operator $E$ based on linear interpolation and
the restriction operator $R$ as the transpose of $E$ after row-normalization. 
Taking the coarse and fine grid nodes shown in Figure \ref{fig:ASN-2} as an example,
the corresponding extension operator $E\in \IR^{16\times 8}$ that interpolates a coarse grid approximation to a fine grid approximation 
has the following matrix formulation (transposed for economical exposition)
\begin{align}
E^\intercal=\bmt
 1 & \frac{2}{3} & \frac{1}{3} & 0 & 0 & 0 & 0 & 0 & 0 & 0 & 0 & 0 & 0 & 0 & 0 & 0\\ 
 0 & \frac{1}{3} & \frac{2}{3} & 1 & 0 & 0 & 0 & 0 & 0 & 0 & 0 & 0 & 0 & 0 & 0 & 0\\ 
 0 & 0 & 0 & 0 & 1 & \frac{2}{3} & \frac{1}{3} & 0 & 0 & 0 & 0 & 0 & 0 & 0 & 0 & 0\\ 
 0 & 0 & 0 & 0 & 0 & \frac{1}{3} & \frac{2}{3} & 1 & 0 & 0 & 0 & 0 & 0 & 0 & 0 & 0\\ 
 0 & 0 & 0 & 0 & 0 & 0 & 0 & 0 & 1 & \frac{2}{3} & \frac{1}{3} & 0 & 0 & 0 & 0 & 0\\ 
 0 & 0 & 0 & 0 & 0 & 0 & 0 & 0 & 0 & \frac{1}{3} & \frac{2}{3} & 1 & 0 & 0 & 0 & 0\\ 
 0 & 0 & 0 & 0 & 0 & 0 & 0 & 0 & 0 & 0 & 0 & 0 & 1 & \frac{2}{3} & \frac{1}{3} & 0\\ 
 0 & 0 & 0 & 0 & 0 & 0 & 0 & 0 & 0 & 0 & 0 & 0 & 0 & \frac{1}{3} & \frac{2}{3} & 1\\
\emt_{8\times 16}.
\end{align}
The corresponding restriction operator $R$ is given by $R=\frac{1}{2} E^\intercal$, with $\frac{1}{2}$ being the normalization factor calculated from the row sum of  $E^\intercal$.
With both $E$ and $R$ in {hand}, one can algebraically {set up} the coarse grid coefficient matrix through the Galerkin projection
\eq
 L_c=RL_hE,
\ee
where  the coarse grid matrix $L_c$ has a much smaller dimension of order $O(K)$.
Similar as the geometric multigrid method, it is also possible to geometrically construct the coarse grid coefficient matrix
by re-discretizing the original continuous problem with a chosen coarse mesh (needs to be uniform in our current finite difference scheme).
Our preliminary numerical tests show that such a geometrical approach of constructing the coarse grid matrix 
and the corresponding extension and restriction operators delivers less robust convergence rates than the above algebraic approach.
A 2-level ASN or MSN algorithm is obtained by complementing the above 1-level ASN or MSN algorithms (either as solver or preconditioner)
with one step of coarse grid correction based on the residual equation.
Suppose we have obtained a 1-level global solution approximation $w_h^{(1)}$ from each iteration of 1-level ASN or MSN algorithms,
the corresponding improved 2-level approximation $w_h^{(2)}$ can be calculated according to the correction
\eq \label{2-level-cc}
 w_h^{(2)}=w_h^{(1)}+E L_c^{-1} R (b_h-L_h w_h^{(1)}),
\ee
where $L_c^{-1}$ should be understood as one efficient coarse grid system solving with certain level of approximation errors.
Here we in fact applied the coarse grid correction in a multiplicative manner \cite{smith1996domain}. 
It is worthwhile to mention that the extra cost of adding one coarse grid correction as in (\ref{2-level-cc})
is usually negligible compared to the overall computational costs, if the number of subdomains is
far less than the total number of time grid points, i.e., $K\ll N$.

Next, we consider the overlapping cases in MSO and ASO algorithms.
The following Figure \ref{fig:ASO-1} depicts a typical overlapping decomposition scheme with $N=17$ and $K=4$,
where each subdomain extends one grid node into its neighboring subdomains in view of the nonoverlapping decompositions
shown in Figure \ref{fig:ASN-1}. Slightly different from the nonoverlapping cases, 
we choose the coarse grid nodes as described in Figure  \ref{fig:ASO-2},
where one grid node centered at the overlap region is selected as a coarse grid node.
Numerical tests show that the current choice gives scalable convergence rates,
although other better choices may not be within our numerical trials.
For instance, we also tested the classical choice of placing one coarse
grid point into the center of each subdomain, which however does not provide very scalable convergence rates.
Further discussion on the construction of better coarse grid spaces is beyond the scope of the current paper
and is left as our future research.
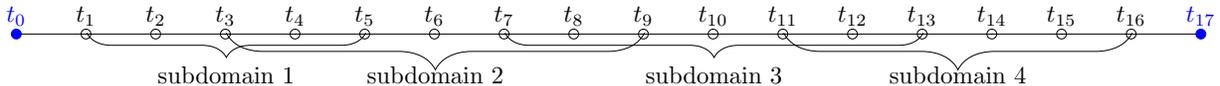
\begin{figure}[!htb] 
    \centering
      \resizebox {\columnwidth} {!} {
 \begin{tikzpicture}[yscale=-1] 
        % 8x8 grid
        \draw (0, 0) grid (17, 0); 
          \draw [color=blue, fill=blue] (0, 0) circle (2pt) node[align=center,   above] {$t_0$};
          \draw [color=blue, fill=blue] (17, 0) circle (2pt) node[align=center,   above] {$t_{17}$};
         \foreach \x in {1,...,16}
        { 
         \draw (\x,0) circle (2pt) node[align=center,   above] {$t_{\x}$};
       }
\draw [decorate,decoration={brace,amplitude=10pt,mirror},xshift=0.4pt,yshift=-0.4pt](1,0) -- (5,0) node[black,midway,yshift=-0.6cm] {subdomain 1};
 \draw [decorate,decoration={brace,amplitude=15pt,mirror},xshift=0.4pt,yshift=-0.4pt](3,0) -- (9,0) node[black,midway,yshift=-0.6cm] {subdomain 2};
\draw [decorate,decoration={brace,amplitude=10pt,mirror},xshift=0.4pt,yshift=-0.4pt](7,0) -- (13,0) node[black,midway,yshift=-0.6cm] {subdomain 3};
 \draw [decorate,decoration={brace,amplitude=15pt,mirror},xshift=0.4pt,yshift=-0.4pt](11,0) -- (16,0) node[black,midway,yshift=-0.6cm] {subdomain 4};
    
    \end{tikzpicture}
    }
    \caption{The overlapping decomposition of all time grid points (empty circles) with $N=17$ and $K=4$.}
      \label{fig:ASO-1}
    \end{figure}
    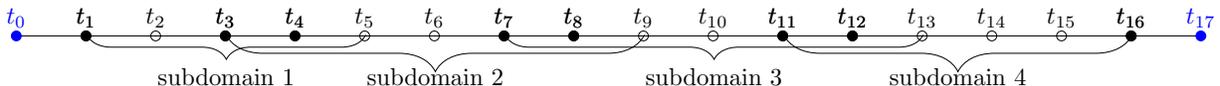
\begin{figure}[!htb] 
    \centering
      \resizebox {\columnwidth} {!} {
 \begin{tikzpicture}[yscale=-1] 
        % 8x8 grid
        \draw (0, 0) grid (17, 0); 
          \draw [color=blue, fill=blue] (0, 0) circle (2pt) node[align=center,   above] {$t_0$};
          \draw [color=blue, fill=blue] (17, 0) circle (2pt) node[align=center,   above] {$t_{17}$};
         \foreach \x in {1,...,16}
        { 
         \draw (\x,0) circle (2pt) node[align=center,   above] {$t_{\x}$};          
        }
         \foreach \x in {1,3,4,7, 8,11,12,16}
        { 
         \filldraw (\x,0) circle (2pt) node[align=center,   above] {$t_{\x}$};                
        }
        
\draw [decorate,decoration={brace,amplitude=10pt,mirror},xshift=0.4pt,yshift=-0.4pt](1,0) -- (5,0) node[black,midway,yshift=-0.6cm] {subdomain 1};
 \draw [decorate,decoration={brace,amplitude=15pt,mirror},xshift=0.4pt,yshift=-0.4pt](3,0) -- (9,0) node[black,midway,yshift=-0.6cm] {subdomain 2};
\draw [decorate,decoration={brace,amplitude=10pt,mirror},xshift=0.4pt,yshift=-0.4pt](7,0) -- (13,0) node[black,midway,yshift=-0.6cm] {subdomain 3};
 \draw [decorate,decoration={brace,amplitude=15pt,mirror},xshift=0.4pt,yshift=-0.4pt](11,0) -- (16,0) node[black,midway,yshift=-0.6cm] {subdomain 4};
   
    \end{tikzpicture}
    }
    \caption{The choice of coarse grid nodes (filled circles) among fine grid points (all circles) with $N=17$ and $K=4$.}
      \label{fig:ASO-2}
    \end{figure}
    
\section{Numerical Examples}\label{sec: num}
%=============================================================
In this section, we provide numerical examples to demonstrate the effectiveness of our proposed methods. 
All simulations are implemented using MATLAB R2016a on a laptop PC
with Intel(R) Core(TM) i5-6200U CPU@2.30GHz and 8GB RAM.
We use the strip decomposition that divides the time interval $[0,T]$ into $K$ uniform parts, thus the space-time domain is partitioned to $K$ chunks. 
On each subdomain, we use the finite difference discretizations developed in \cite{LX2015a}, which provides a second-order accuracy in both space and time. 
The time step size $\tau$ is taken to be same as the spatial mesh size $h$.
For 1D examples, the sub-domain systems are solved by sparse direct solvers, such as the backslash `mldivide' solver of MATLAB.
For 2D examples, the efficient semi-coarsening multigrid solver with one V-cycle \cite{LX2015a} is utilized to approximately solve the sub-linear system on each subdomain.
However, any other efficient iterative solvers, such as algebraic multigrid methods \cite{Treister_2015}, are also applicable.

When overlapping methods are used, each subdomain is extended and overlaps its neighbors by a minimal overlap $\delta=\tau$  in the temporal direction. 
All iterative algorithms start with a random initial guess (use {the} \textsf{rand} function in MATLAB). 
Notice that our domain decomposition algorithms can be used as stand-alone iterative solvers and preconditioners, respectively.
We choose the standard stopping condition based on relative reduction of residual norms, i.e.,
\[
 \|r_k\|/\|r_0\|<10^{-7},
\]
where $r_k$ is the global residual vector at $k$-th iteration and $r_0$ denotes the initial residual vector.  
%Based on the level of discretization errors, we choose $tol=10^{-7}$ and $tol=10^{-6}$ for 1D examples and 2D examples, respectively.
We will use right-preconditioned GMRES and choose the same stopping criterion as the stand-alone solvers.
% In addition to this condition, due to discretization errors, when the algorithms act as stand-alone iterative solvers, 
% the iteration will also stop 
% when the norm difference between approximations of two successive iterations satisfies $\|y^{k}-y\|+\|p^{k}-p\|<10^{-6}$.
%This choice is sufficient to recover the expected second-order accuracy without performing too many iterations.

\subsection{1D Examples with Direct Subdomain Solvers} \label{num1D}
We first study one-dimensional examples to numerically verify our proposed methods.
In 1D cases, all sub-domain systems are accurately solved up to a machine precision,
with the approximation errors in solving sub-domain systems far less than {those}
caused by the domain decomposition iterations. 
This allows us to focus on evaluating the convergence performance of the domain decomposition algorithms themselves.

\textbf{Example 1.} Let $\Omega=(0,1)$, $T=4$, and $\gamma=10^{-2}$.
%The parabolic optimal control problem \eqref{goal}-\eqref{state} 
%with a regularization parameter $\gamma=10^{-2}$ is considered. 
Choose $f$, $g$, and $y_0$ in (\ref{opt1A}) so that the exact solution is given by
$$y(x,t)=\cos(\pi t)\sin(\pi x_1)
 \tn{and}
p(x,t)=\sin(\pi t)\sin(\pi x_1).$$

\begin{figure}[H]
\centering 
\includegraphics[width=0.45\textwidth]{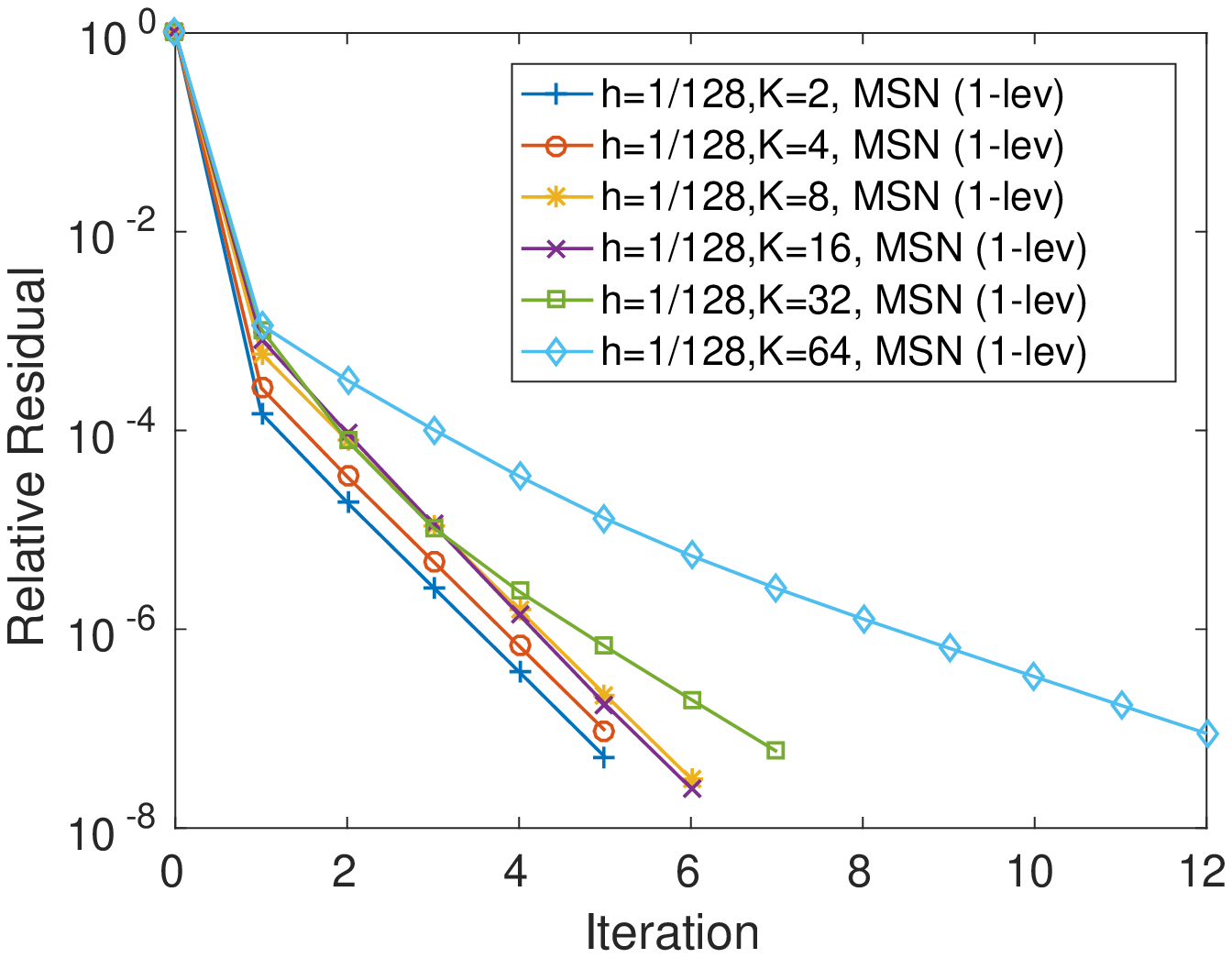}
\includegraphics[width=0.45\textwidth]{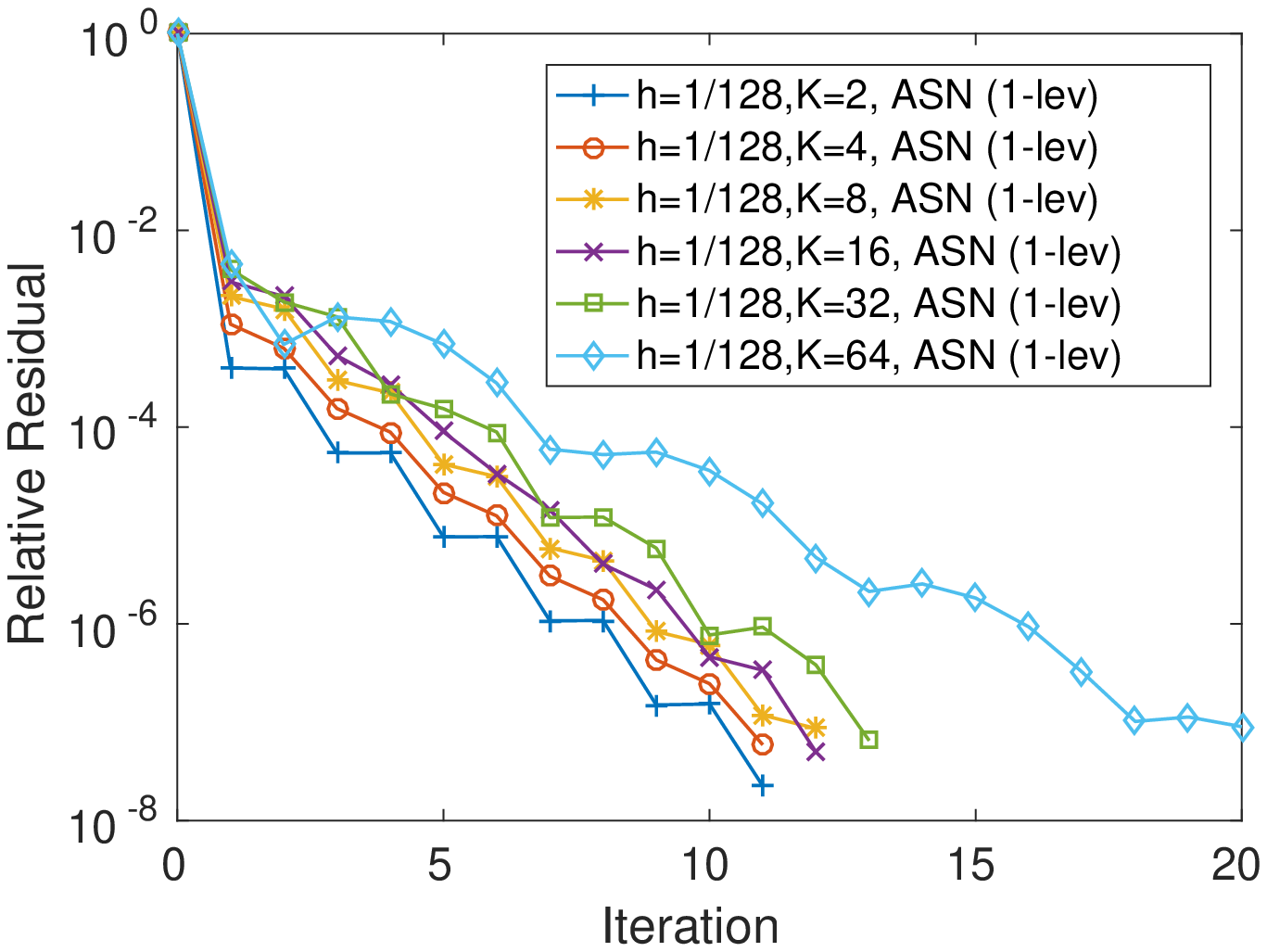}
%\end{minipage}
\caption{Convergence of 1-level MSN and ASN as stand-alone solvers. }
\label{fig:ResMSNSolver}
\end{figure}
\begin{figure}[H]
\centering 
\includegraphics[width=0.45\textwidth]{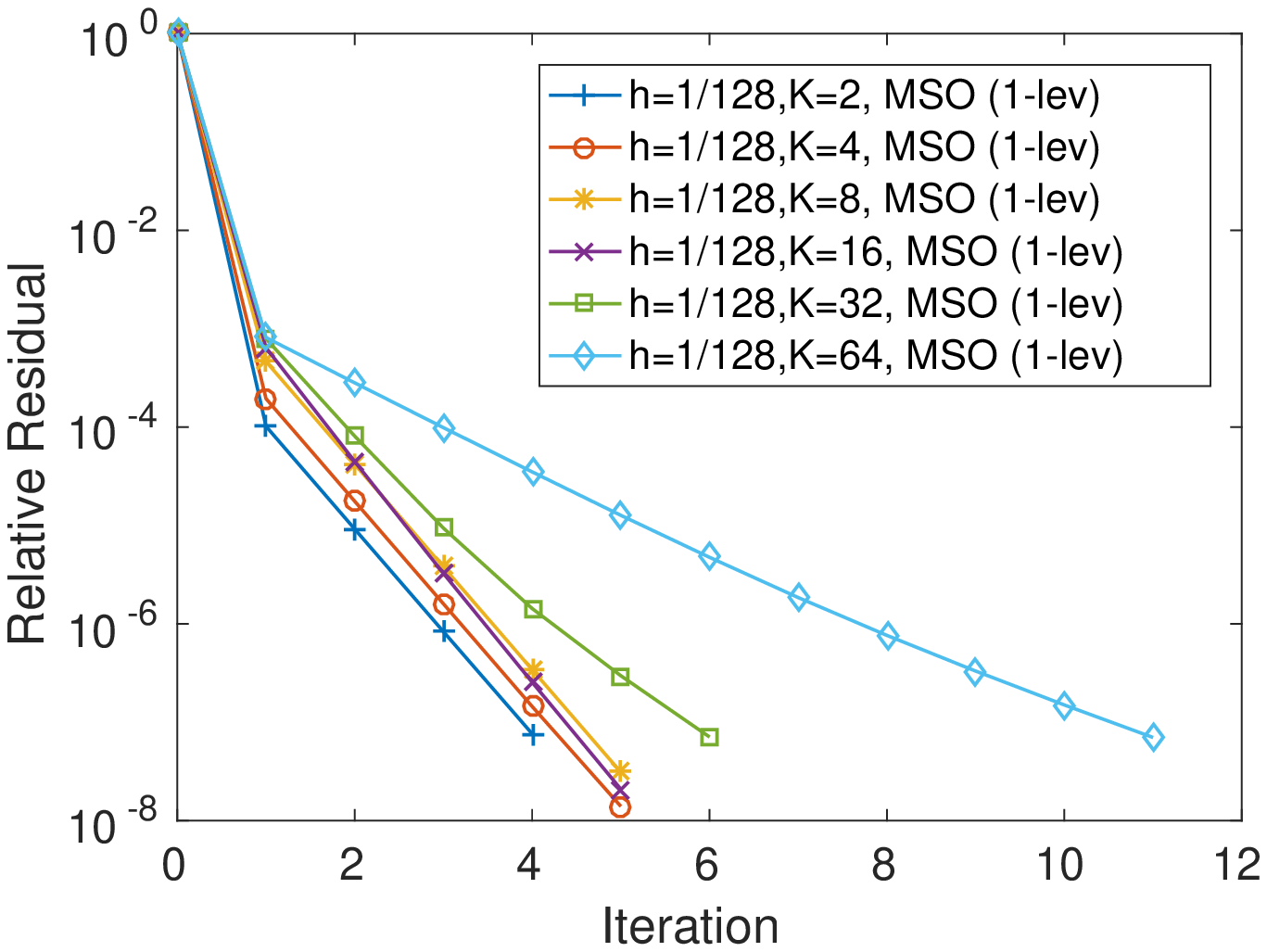}
\includegraphics[width=0.45\textwidth]{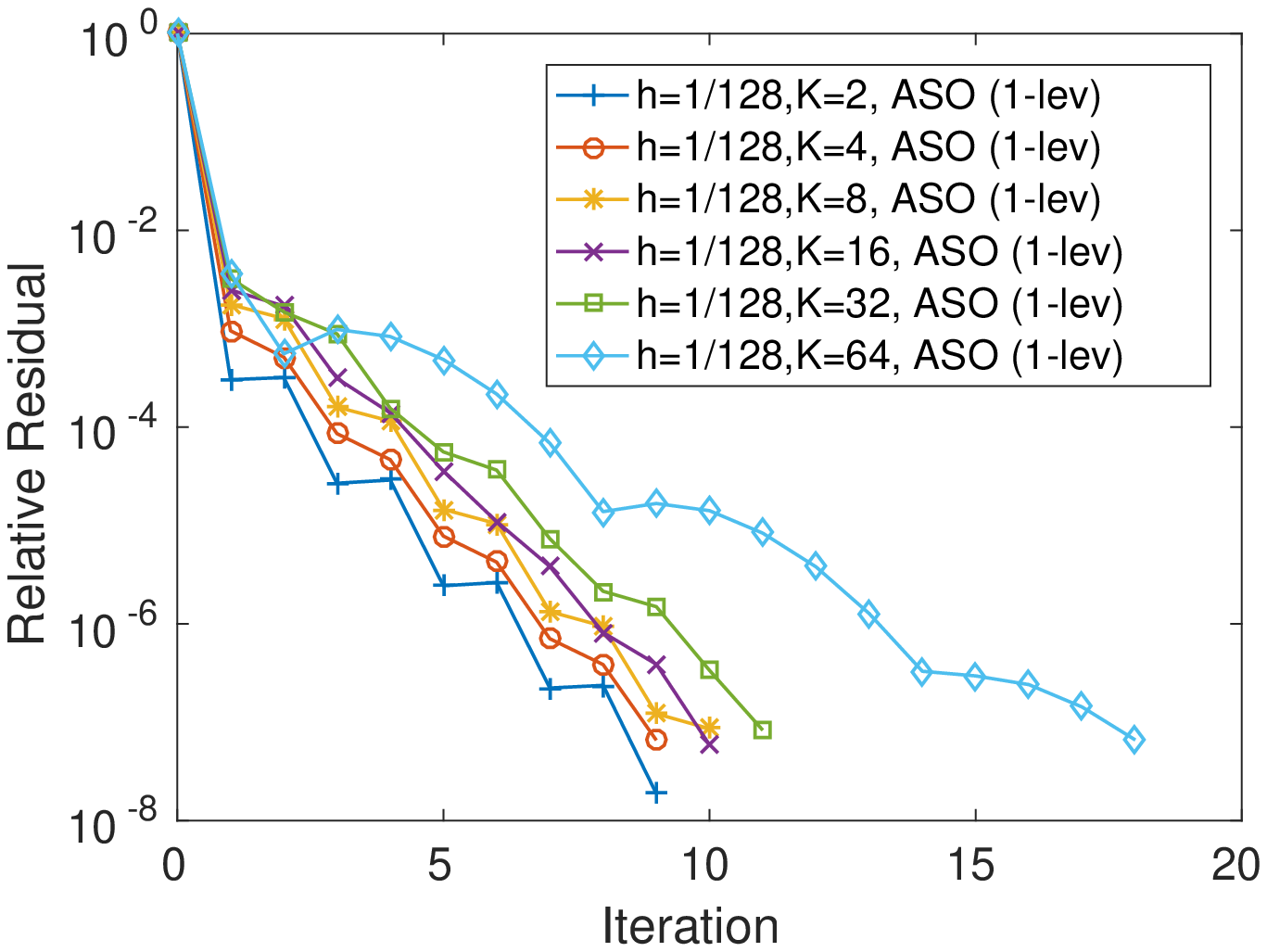}
%\end{minipage}
\caption{Convergence of 1-level MSO and ASO as stand-alone solvers. }
\label{fig:ResMSOSolver}
\end{figure}

In Figure \ref{fig:ResMSNSolver} and Figure \ref{fig:ResMSOSolver}, 
we show the convergence of 1-level  MSN, ASN, MSO, and ASO algorithms
as stand-alone solvers, respectively. 
We observe that (i) the multiplicative solvers (MSN, MSO) require less iterations and therefore converge faster than the corresponding additive solvers (ASN, ASO); 
 (ii) for both multiplicative and additive solvers, 
the corresponding overlapping solvers (MSO, ASO) have slightly faster convergence than the nonoverlapping ones (MSN, ASN) 
as the number of subdomains $K$ is increased;
Nevertheless, for all the solvers, the required iteration numbers show evident growth as the number of subdomains $K$ increases up to $64$.
{This} is expected for such 1-level algorithms since more subdomains imply that more iterations are needed to exchange new approximations 
computed locally on one subdomain to all the other subdomains.
In both figures, we observe that multiplicative algorithms converge monotonically,
while additive algorithms show certain irregular non-monotonic decreasing in the computed residual norms.
Such a difference can be briefly explained by a simple comparison of the proofs in both Theorem \ref{thmASN} and \ref{thmMSN},
where the MSN algorithm has strictly monotonic decrease in both error term $e_1^{k}(T_1)$ and $w_2^{k}(T_1)$ individually
and the ASN algorithm only assures a strictly monotonic decrease in the weighted error term 
$(\gamma \|e_1^{k}(T_1)\|^2+\|w_2^{k}(T_1)\|^2)$.
Especially, when $\gamma$ is small, its scaling effect on the global system with $1/\gamma$ may lead to
our observed non-monotonic decreases of residual norms in the ASN algorithm as stand-alone solvers.
{Furthermore, we report in Table \ref{T1A} the required iteration numbers and CPU time (in seconds),
where the corresponding computation time is based on our sequential MATLAB codes.
We would expect significant reduction in computation time with respect to $K$ when our algorithms are implemented in parallel codes.}

\begin{table}[h!] 
\centering
\caption{{The iteration numbers and CPU time of 1-level MSN, ASN, MSO, and ASO as stand-alone solvers ($h=1/128$).}}
\begin{tabular}{|c|cc|cc|cc|cc|}\hline 
&\multicolumn{2}{c}{MSN}&\multicolumn{2}{|c|}{ASN} 
&\multicolumn{2}{c}{{MSO}}&\multicolumn{2}{|c|}{ASO}\\
\hline
$K$&Iter&CPU&Iter&CPU&Iter&CPU&Iter&CPU \\ \hline
 2&	 5&	3.661   &	 11&	7.933    &	 4&	3.404    &	 9&	    7.237 \\
4&	 5&	4.053   &	 11&	8.422    &	 5&	3.325    &	 9&	    5.923 \\
8&	 6&	3.334   &	 12&	6.577    &	 5&	2.868    &	 10&	5.702 \\
16&	 6&	3.079   &	 12&	6.082    &	 5&	2.719    &	 10&	5.512 \\
32&	 7&	3.102   &	 13&	5.628    &	 6&	3.011    &	 11&	5.734 \\
64&	 12&	4.319   &	 20&	7.088    &	 11&	4.995    &	 18&	8.456 \\
\hline
\end{tabular}
\label{T1A}
 \end{table}

\begin{table}[h!] 
\centering
\caption{{The iteration numbers and CPU time of 2-level MSN, ASN, MSO, and ASO as stand-alone solvers ($h=1/128$).}}
\begin{tabular}{|c|cc|cc|cc|cc|}\hline 
&\multicolumn{2}{c}{MSN}&\multicolumn{2}{|c|}{ASN} 
&\multicolumn{2}{c}{{MSO}}&\multicolumn{2}{|c|}{ASO}\\
\hline
$K$&Iter&CPU&Iter&CPU&Iter&CPU&Iter&CPU \\ \hline
 2&	 6&	4.464    &	 10&7.287    &	 5&	4.174    &	 8&	6.753 \\
4&	 6&	4.950    &	 10&7.921    &	 5&	3.420    &	 8&	5.444 \\
8&	 8&	4.607    &	 10&6.056    &	 6&	3.561    &	 9&	5.402 \\
16&	 8&	4.455    &	 10&5.840    &	 7&	4.076    &	 8&	4.772 \\
32&	 8&	4.449    &	 9&	4.982    &	 6&	3.579    &	 8&	4.990 \\
64&	 6&	3.950    &	 7&	4.556    &	 5&	3.695    &	 6&	4.543 \\
\hline
\end{tabular}
\label{T1B}
 \end{table}

{To obtain a} more robust convergence rate that is independent of the number of subdomains $K$,
2-level algorithms based on some appropriately chosen coarse space correction are often utilized to suppress the increasing iteration numbers.
In Figure \ref{fig:ResMSNSolver2} and Figure \ref{fig:ResMSOSolver2}, 
we show the convergence of our implemented 2-level  MSN, ASN, MSO, and ASO algorithms
as stand-alone solvers, respectively. 
Here we used a standard coarse level space of a small size $K$ based {on} linear interpolation under the framework of Galerkin method, 
as discussed in \cite{Martin2012,Gander2014}. 
Compared to the above 1-level algorithms, we do observe that the required iteration numbers
are almost independent of the increasing number of subdomains $K$, which is critical to the overall performance of the corresponding massively parallel implementations.
Moreover, a coarse space correction step seems to be more effective in improving additive algorithms than the multiplicative ones.
{We also report in Table \ref{T1B} the required iteration numbers and sequential CPU time (in seconds).
In general, two-level algorithms are more efficient than one-level ones as the number of subdomains gets larger.
Even in the setting of sequential computing, we do observe the advantage of 2-level algorithms over 1-level algorithms.}

\begin{figure}[H]
\centering 
\includegraphics[width=0.45\textwidth]{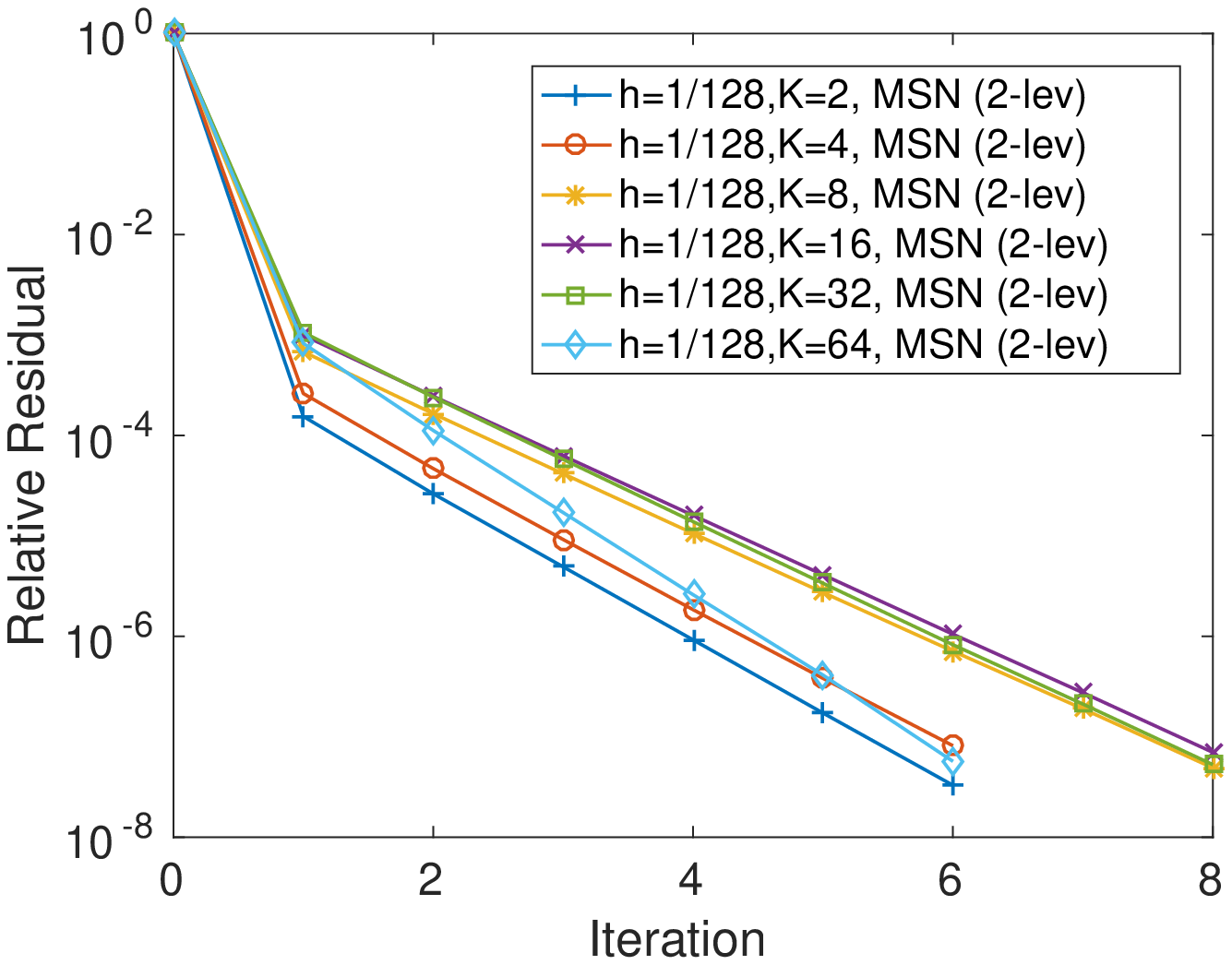}
\includegraphics[width=0.45\textwidth]{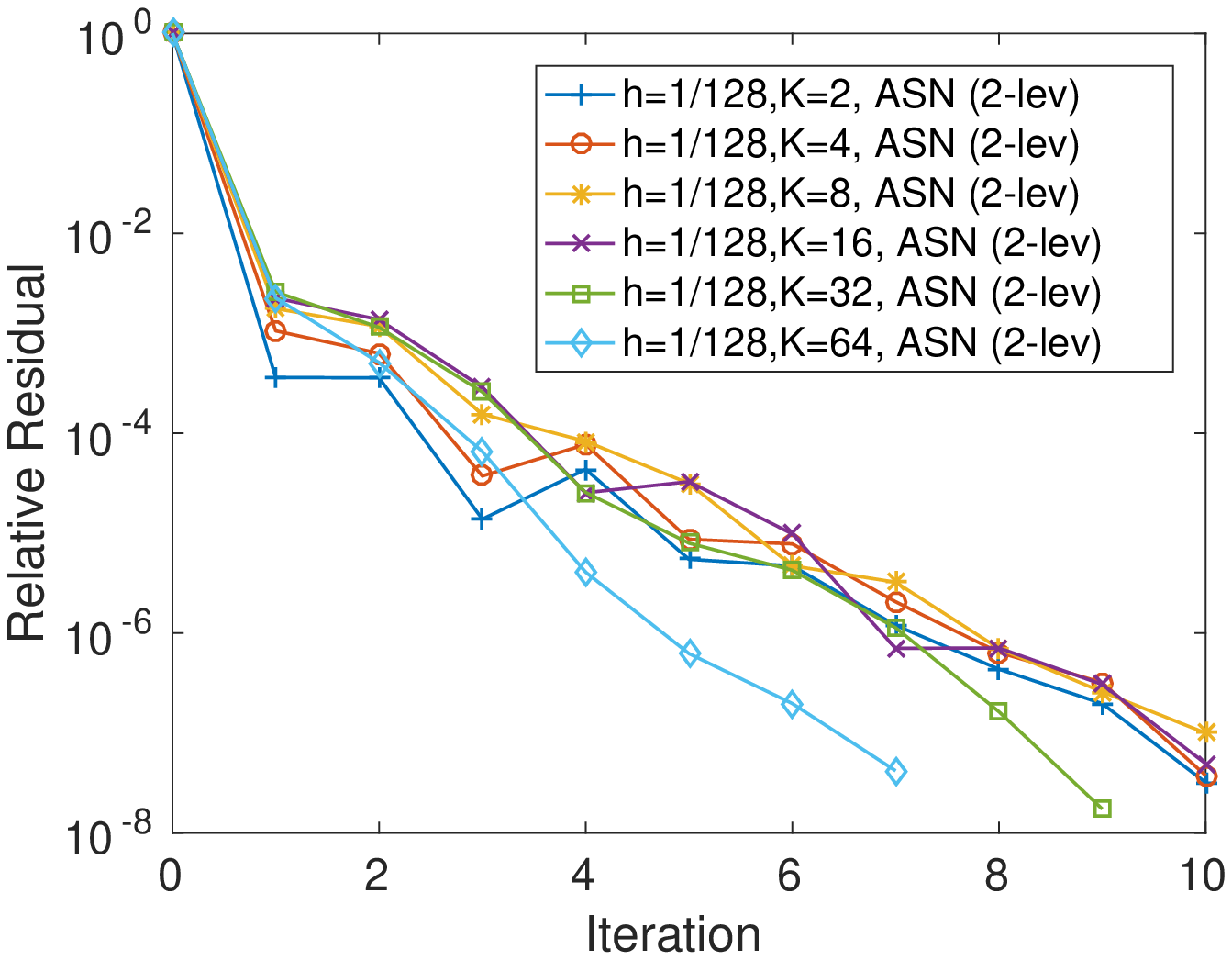}
%\end{minipage}
\caption{Convergence of 2-level MSN and ASN as stand-alone solvers. }
\label{fig:ResMSNSolver2}
\end{figure}
\begin{figure}[H]
\centering 
\includegraphics[width=0.45\textwidth]{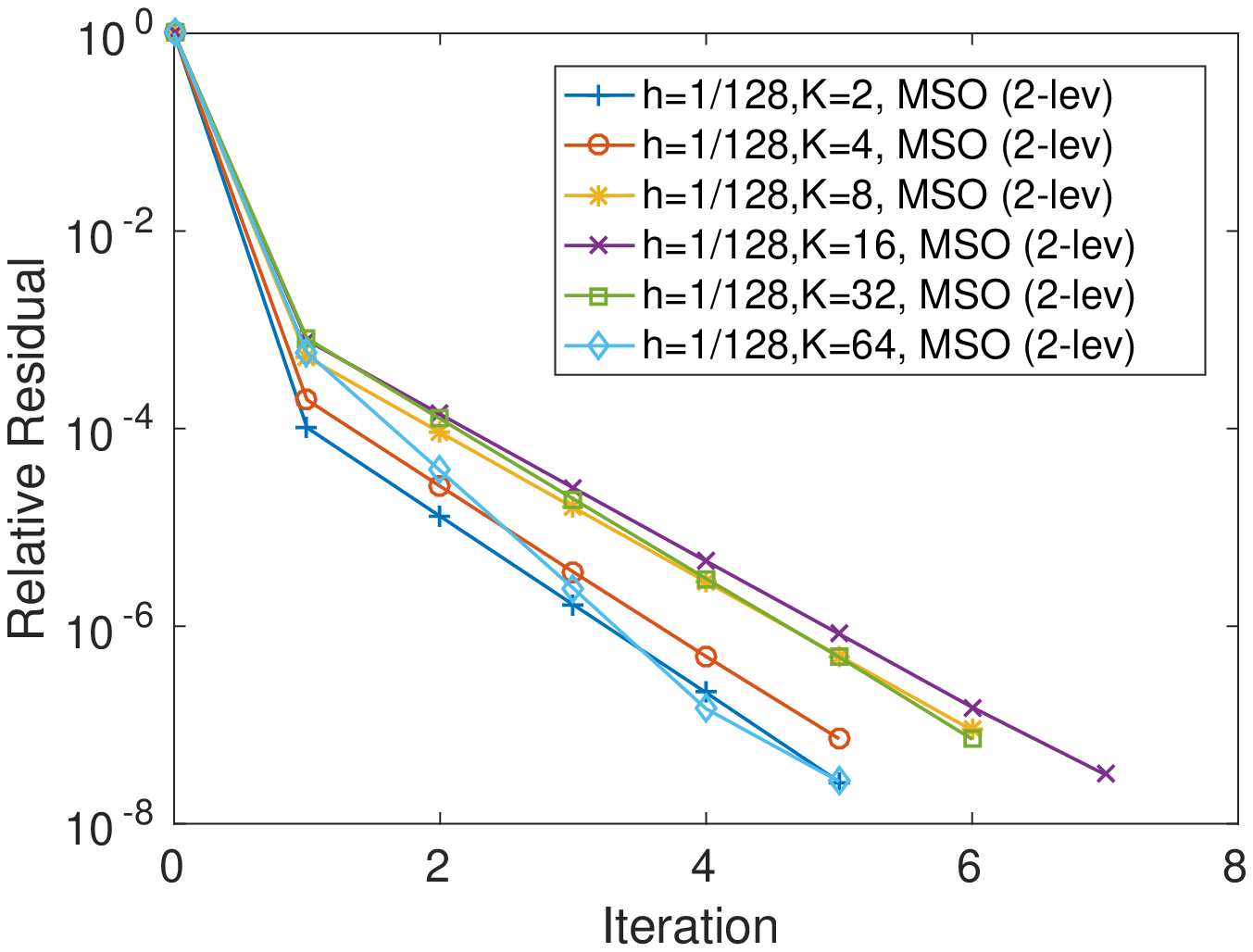}
\includegraphics[width=0.45\textwidth]{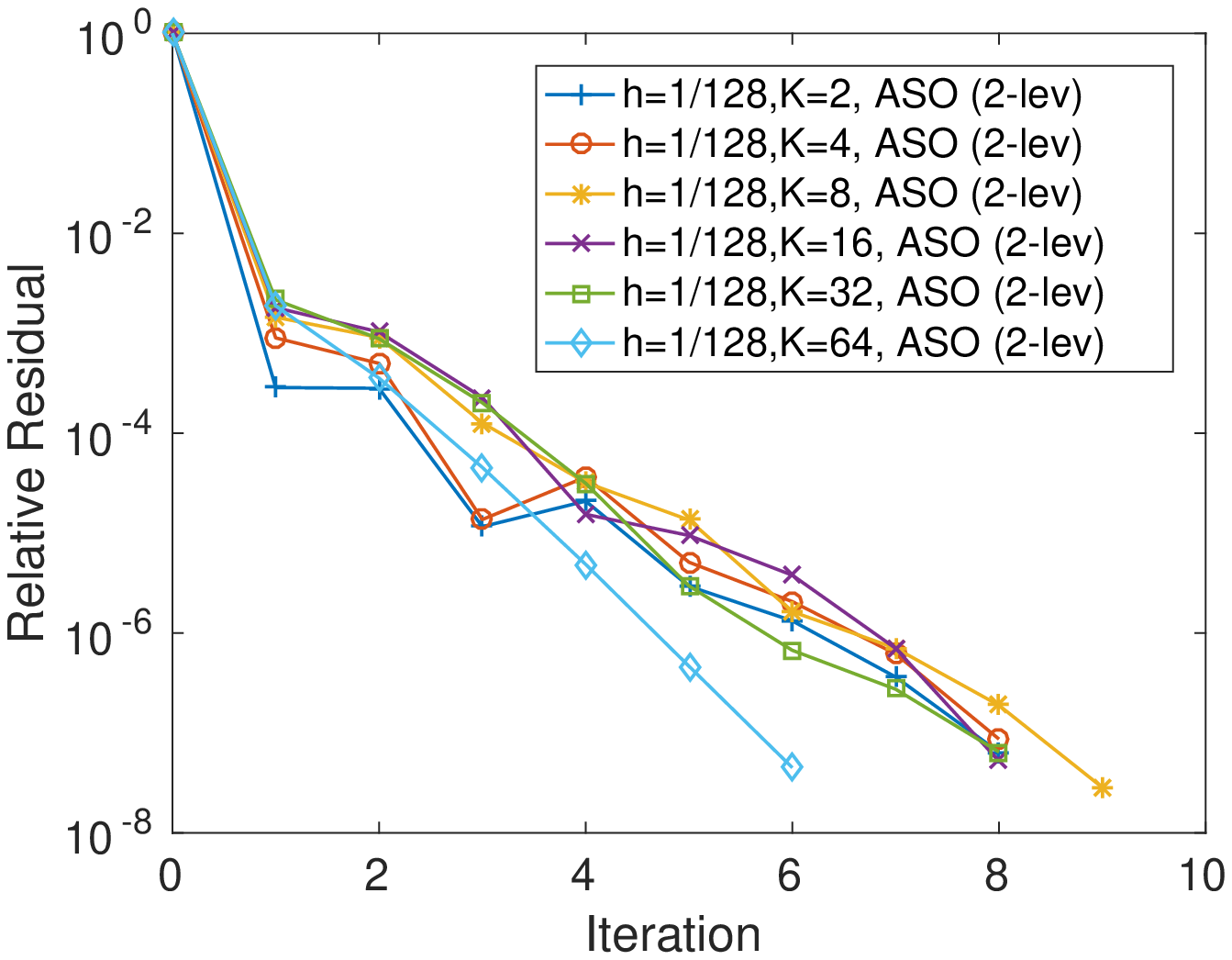}
%\end{minipage}
\caption{Convergence of 2-level MSO and ASO as stand-alone solvers. }
\label{fig:ResMSOSolver2}
\end{figure}

Although theoretically sound, the above illustrated 1-level and 2-level algorithms as stand-alone solvers are seldomly used in practice,
since the overall convergence rate is very problem-dependent and sensitive to the choice of coarse level space.
If the coarse space is not appropriately chosen, the stand-alone solver may become inefficient or even divergent.
Therefore, these iterative algorithms are always used as effective preconditioners for some appropriate Krylov subspace methods, which often
demonstrate more robust convergence rates than as stand-alone solvers.
In Figure \ref{fig:ResPrecASN1}, we plot the convergence of GMRES with no preconditioner (No Prec.) 
and with our proposed 1-level ASN preconditioner (with $K=16$), respectively, for various refined mesh step sizes.
Without any preconditioner, the GMRES demonstrates a clearly deteriorated convergence rate as the step size $h=\tau$ is halved,
which in fact leads to a four times larger system of linear equations with roughly four time larger condition number.
However, the GMRES with {a} 1-level ASN preconditioners shows mesh-independent convergence,
which {one} attributes to the corresponding 1-level ASN stand-alone solver having a mesh-independent convergence rate.
In contrast to the possible non-monotonic decay of residual norms in the ASN algorithm as a stand-alone solver,
the preconditioned GMRES solver yields a robust monotonic decaying residual norms.
\begin{figure}[H]
\centering 
\includegraphics[width=0.45\textwidth]{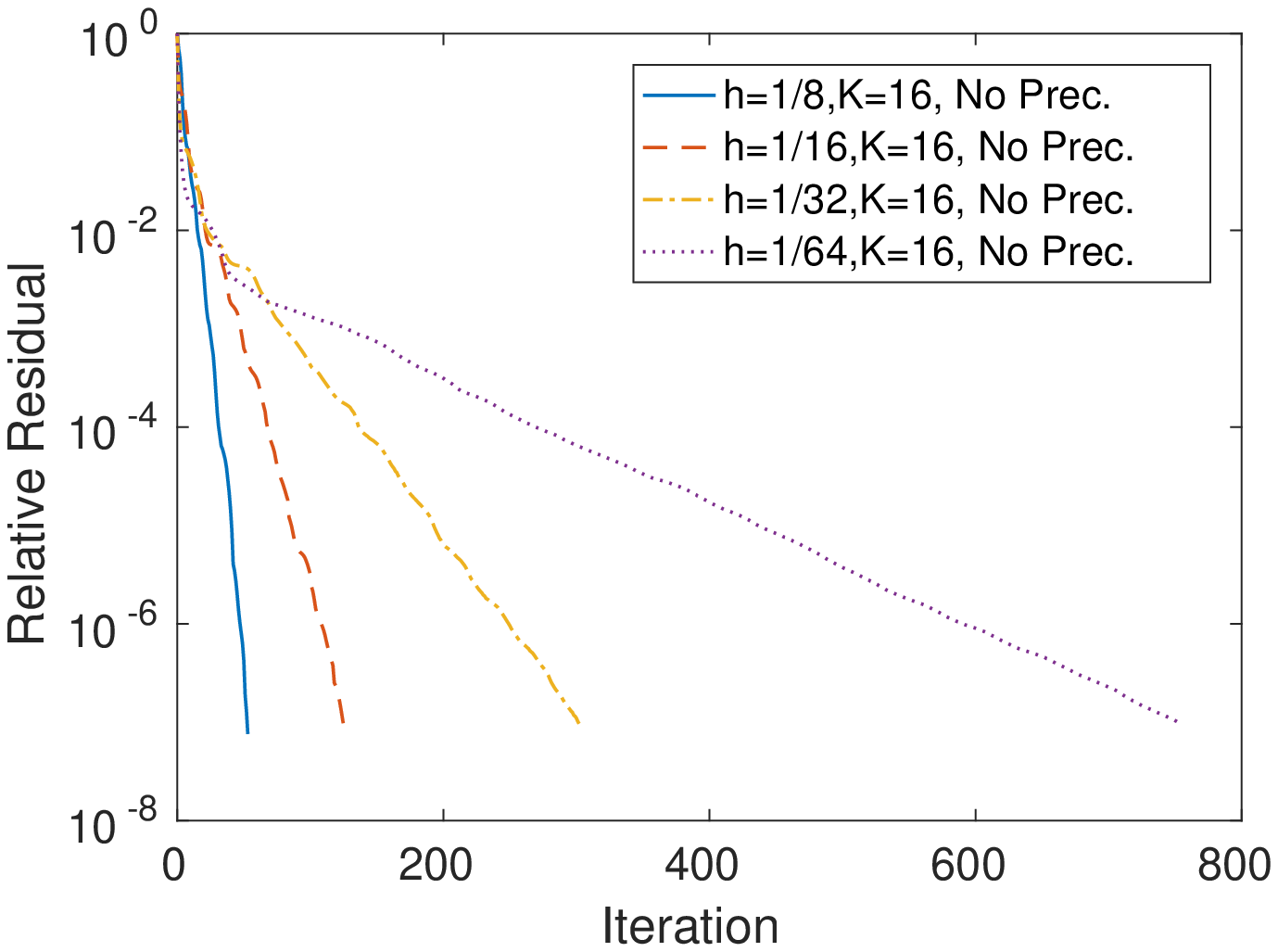}
\includegraphics[width=0.45\textwidth]{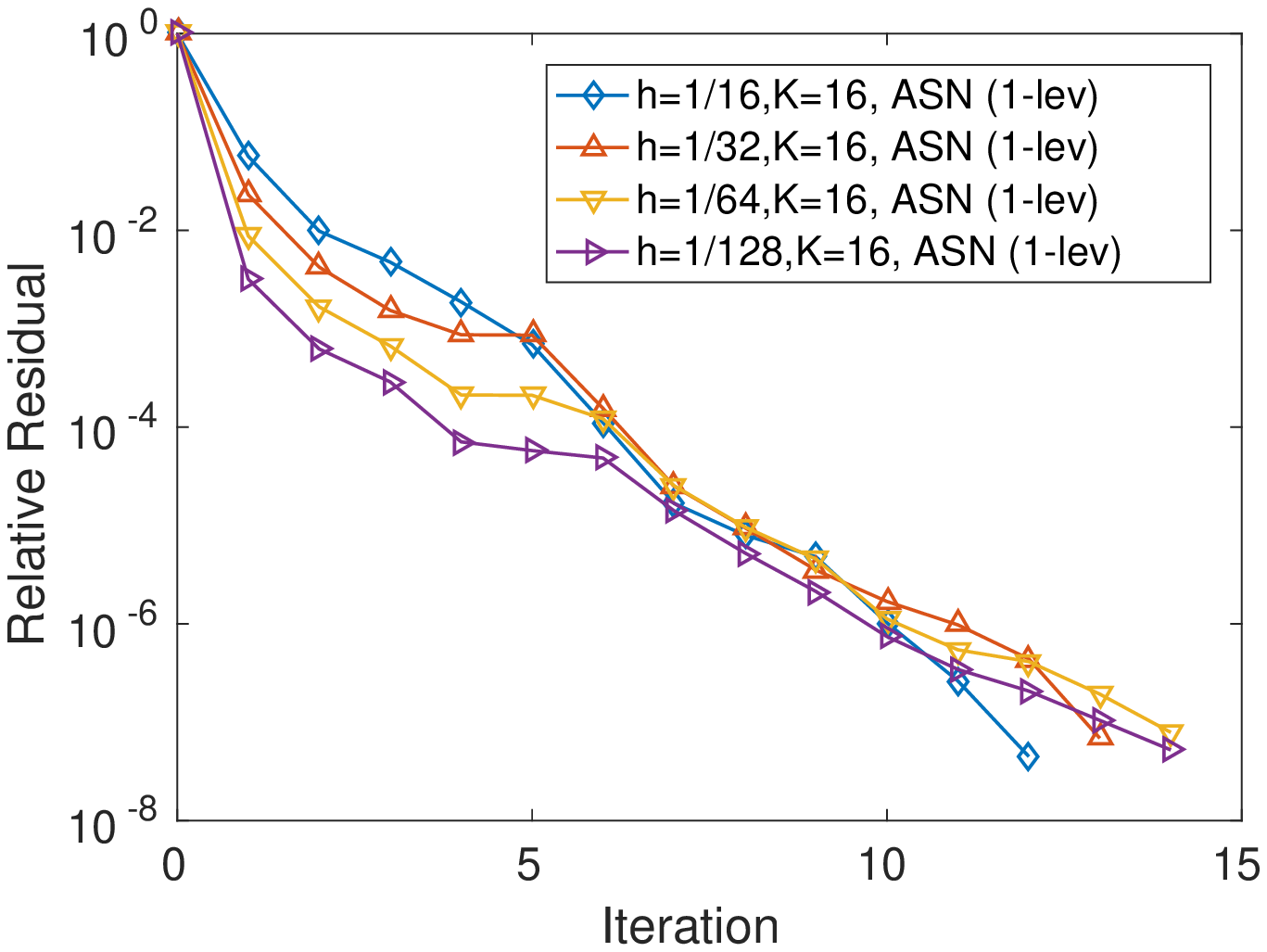}
%\end{minipage}
\caption{Convergence of GMRES with no preconditioner and 1-level ASN ($K=16$) preconditioners, respectively. }
\label{fig:ResPrecASN1}
\end{figure}

We further examine the convergence performance of such 1-level domain decomposition preconditioners with respect to $K$. 
In Figures \ref{fig:ResMSNPre}--\ref{fig:ResMSOPre}, 
we plot the convergence of right-preconditioned GMRES with 1-level MSN, ASN, MSO, and ASO algorithms as preconditioners for $K$ different numbers of subdomains, respectively.
Similar to what we have observed in 1-level stand-alone solvers, the convergence rates of the preconditioned GMRES with 1-level preconditioners are quite satisfactory for small $K$,
but all of them become evidently slower as the number of subdomains $K$ is increased above 16.
\begin{figure}[H]
\centering 
\includegraphics[width=0.45\textwidth]{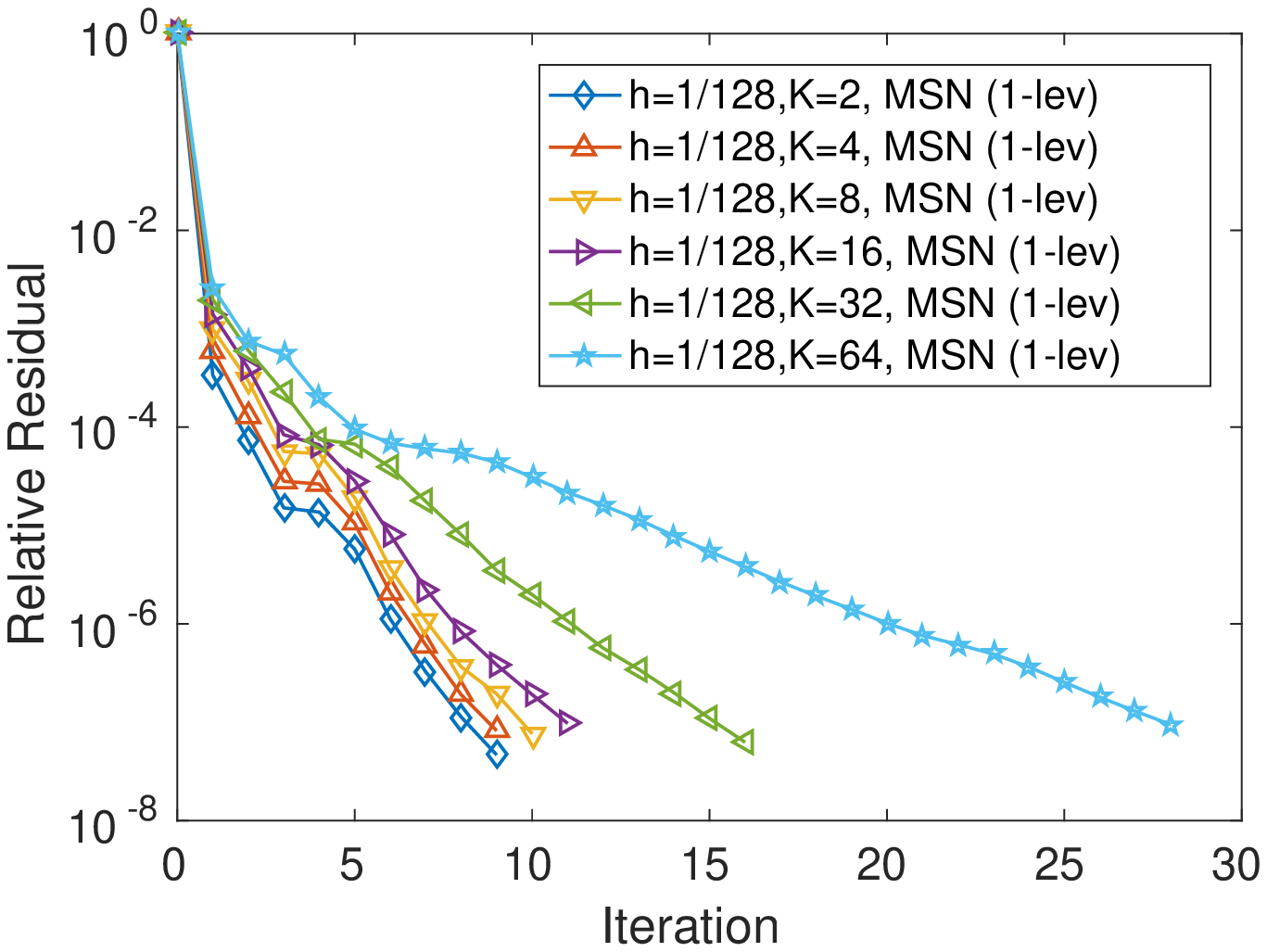}
\includegraphics[width=0.45\textwidth]{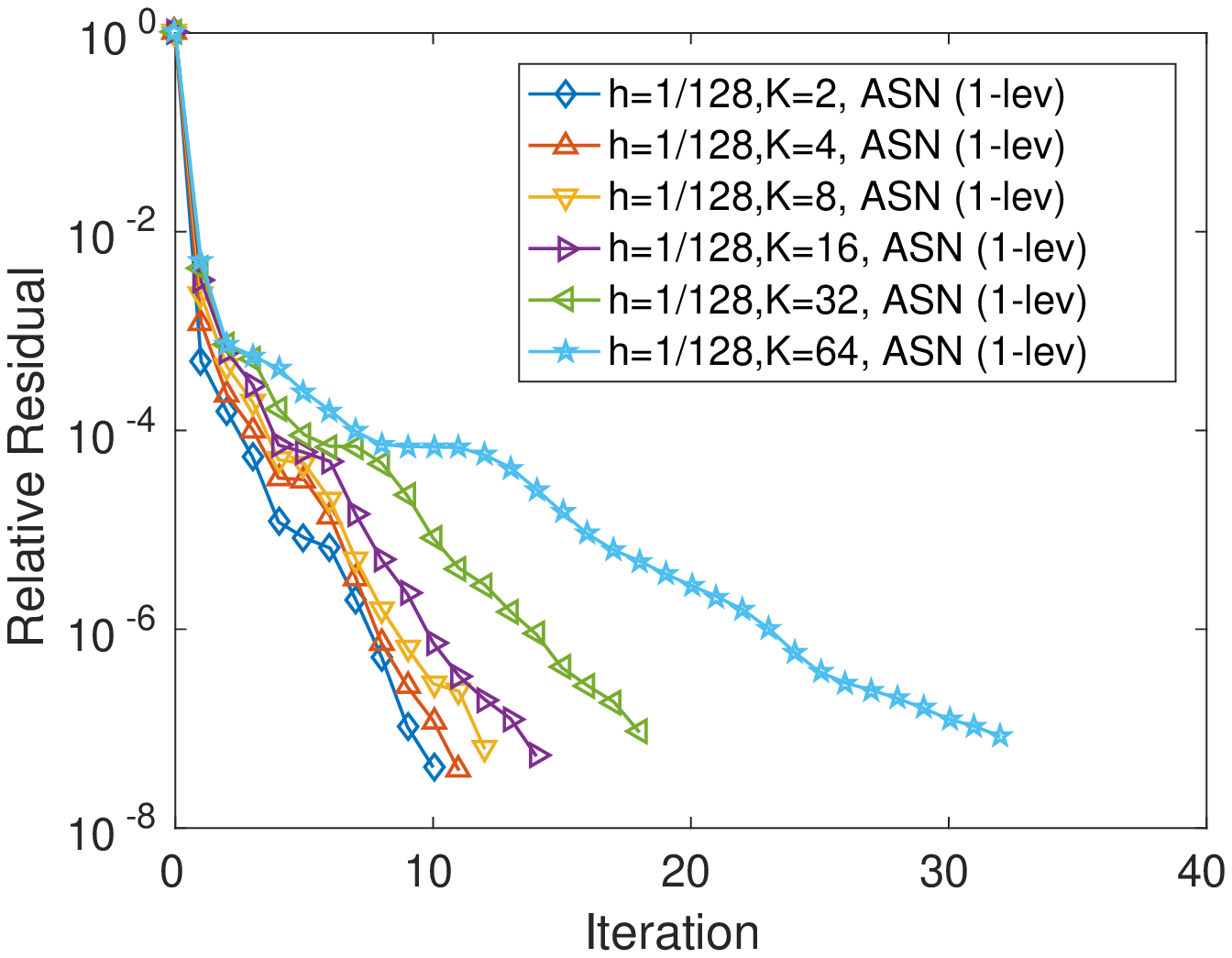}
%\end{minipage}
\caption{Convergence of GMRES with 1-level MSN and ASN preconditioners, respectively. }
\label{fig:ResMSNPre}
\end{figure}
\begin{figure}[H]
\centering 
\includegraphics[width=0.45\textwidth]{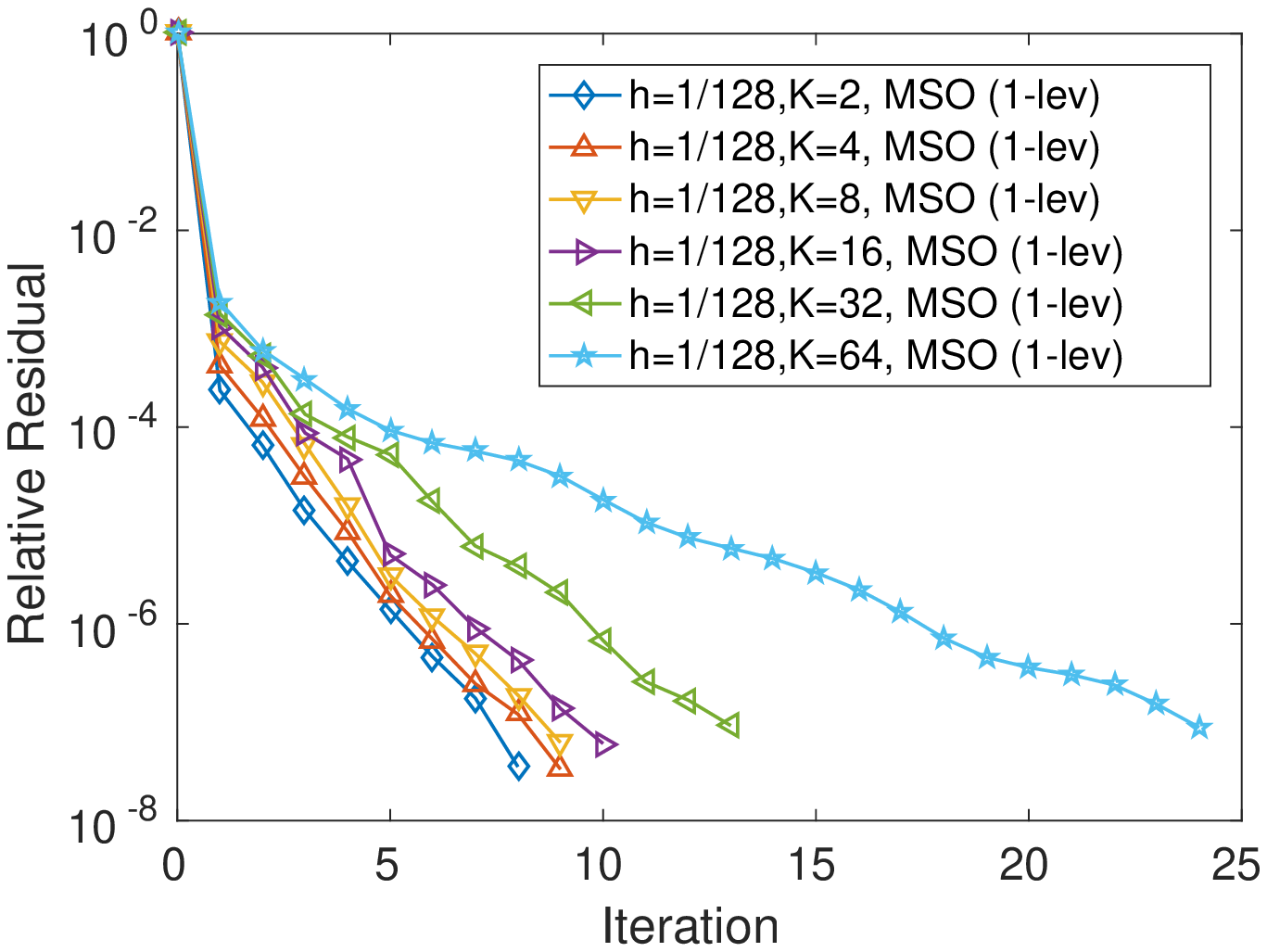}
\includegraphics[width=0.45\textwidth]{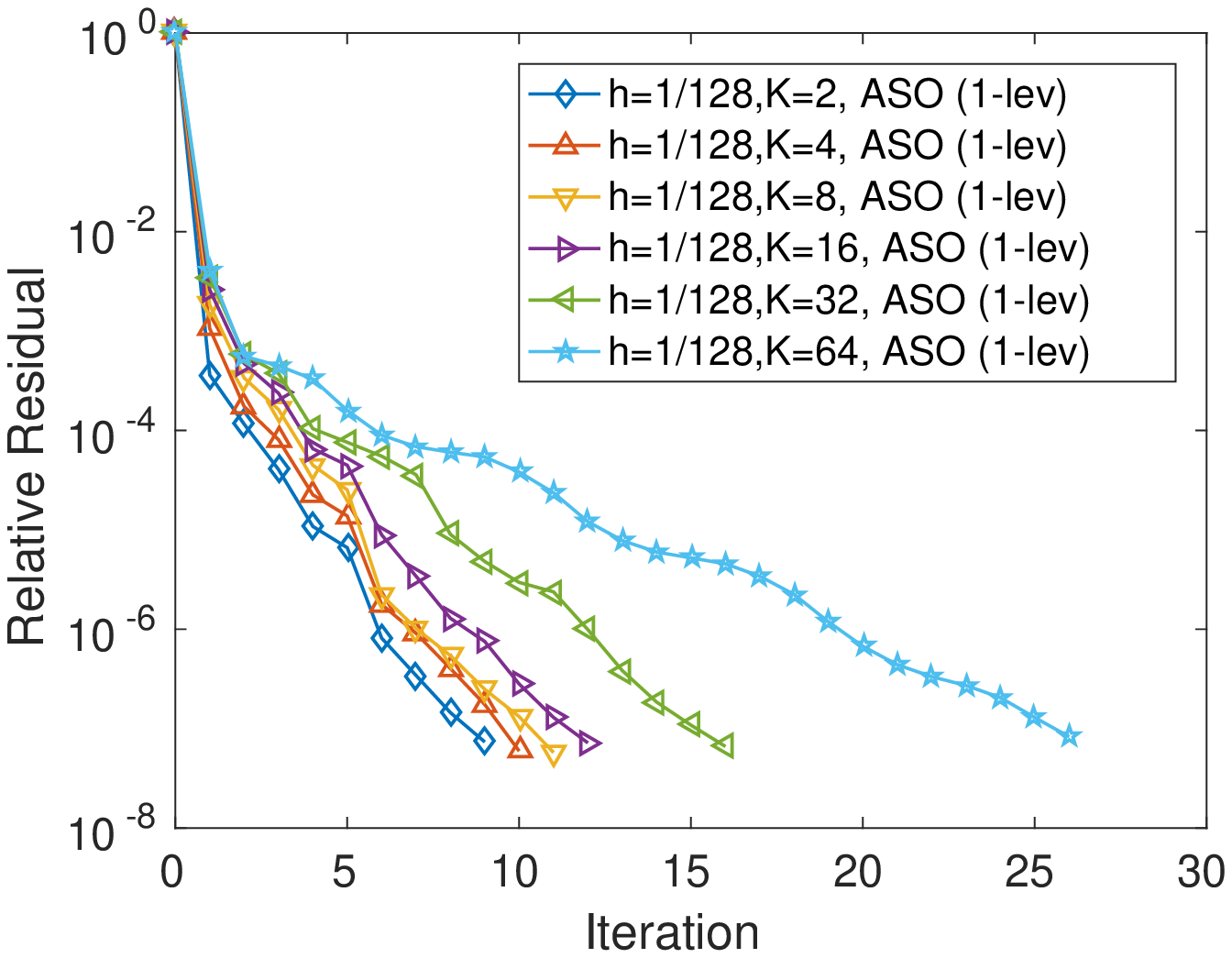} 
\caption{Convergence of GMRES with 1-level MSO and ASO preconditioners, respectively. }
\label{fig:ResMSOPre}
\end{figure}
\begin{figure}[H]
\centering 
\includegraphics[width=0.45\textwidth]{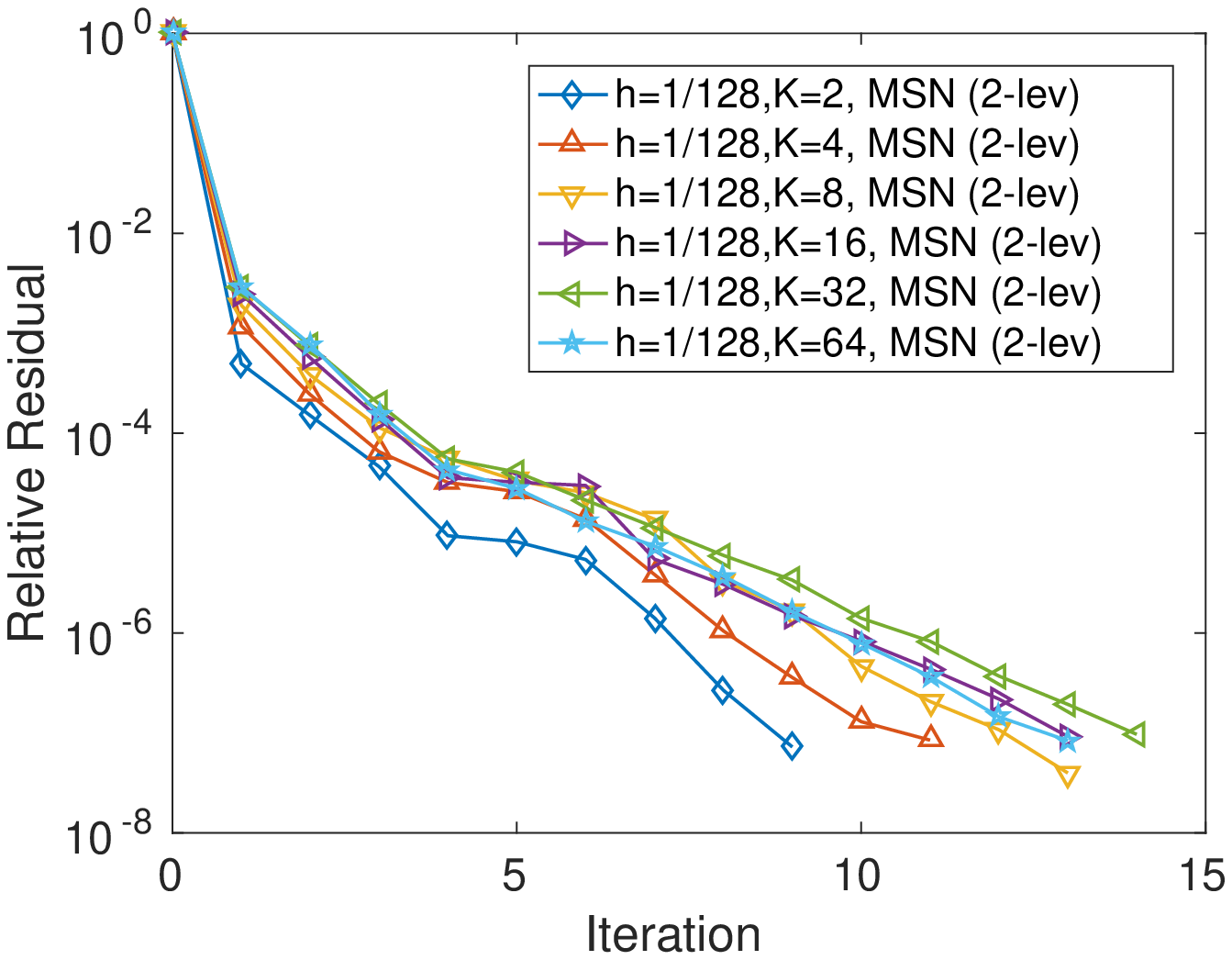}
\includegraphics[width=0.45\textwidth]{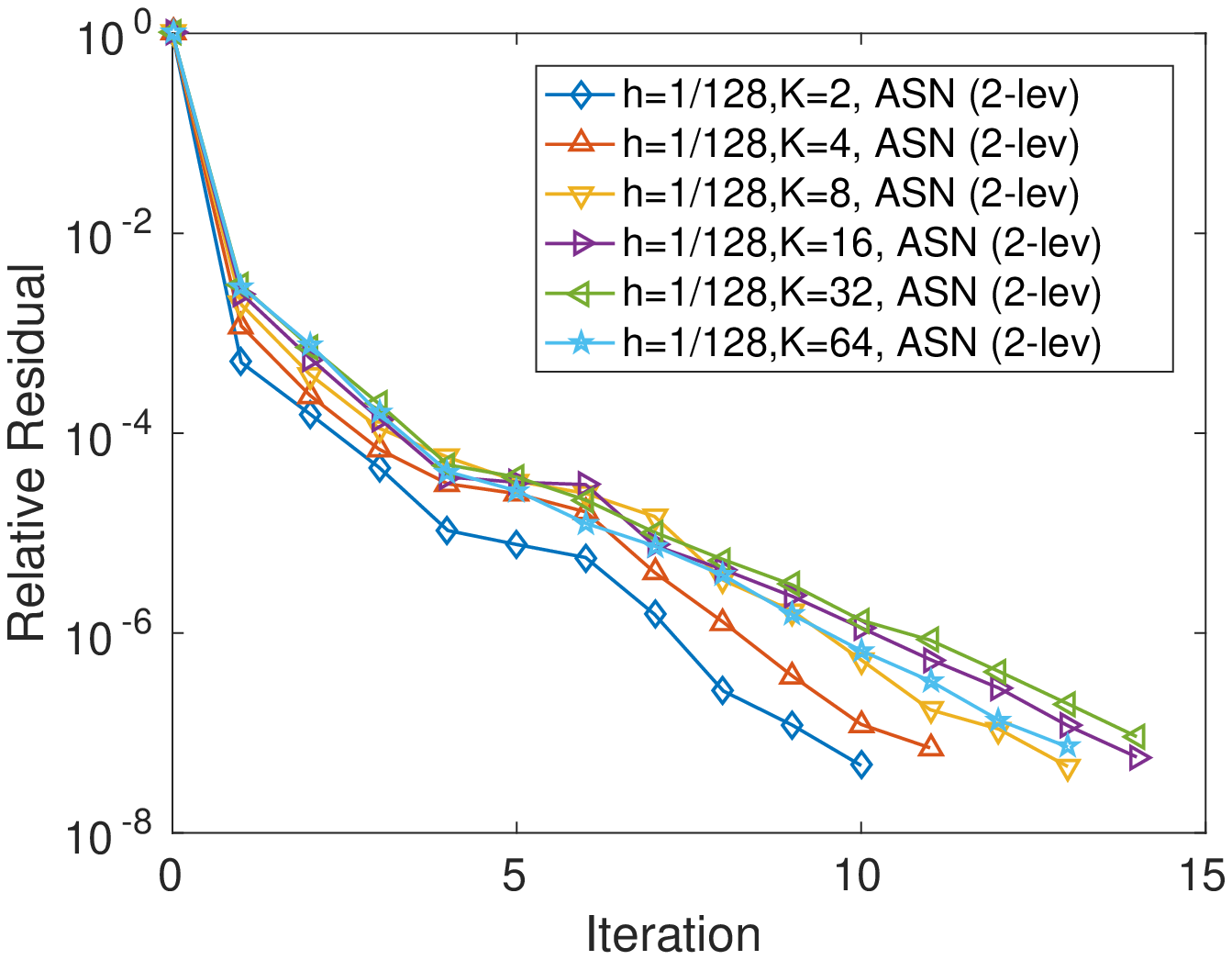} 
\caption{Convergence of GMRES with 2-level MSN and ASN preconditioners, respectively. }
\label{fig:ResMSNPre2}
\end{figure}
\begin{figure}[H]
\centering 
\includegraphics[width=0.45\textwidth]{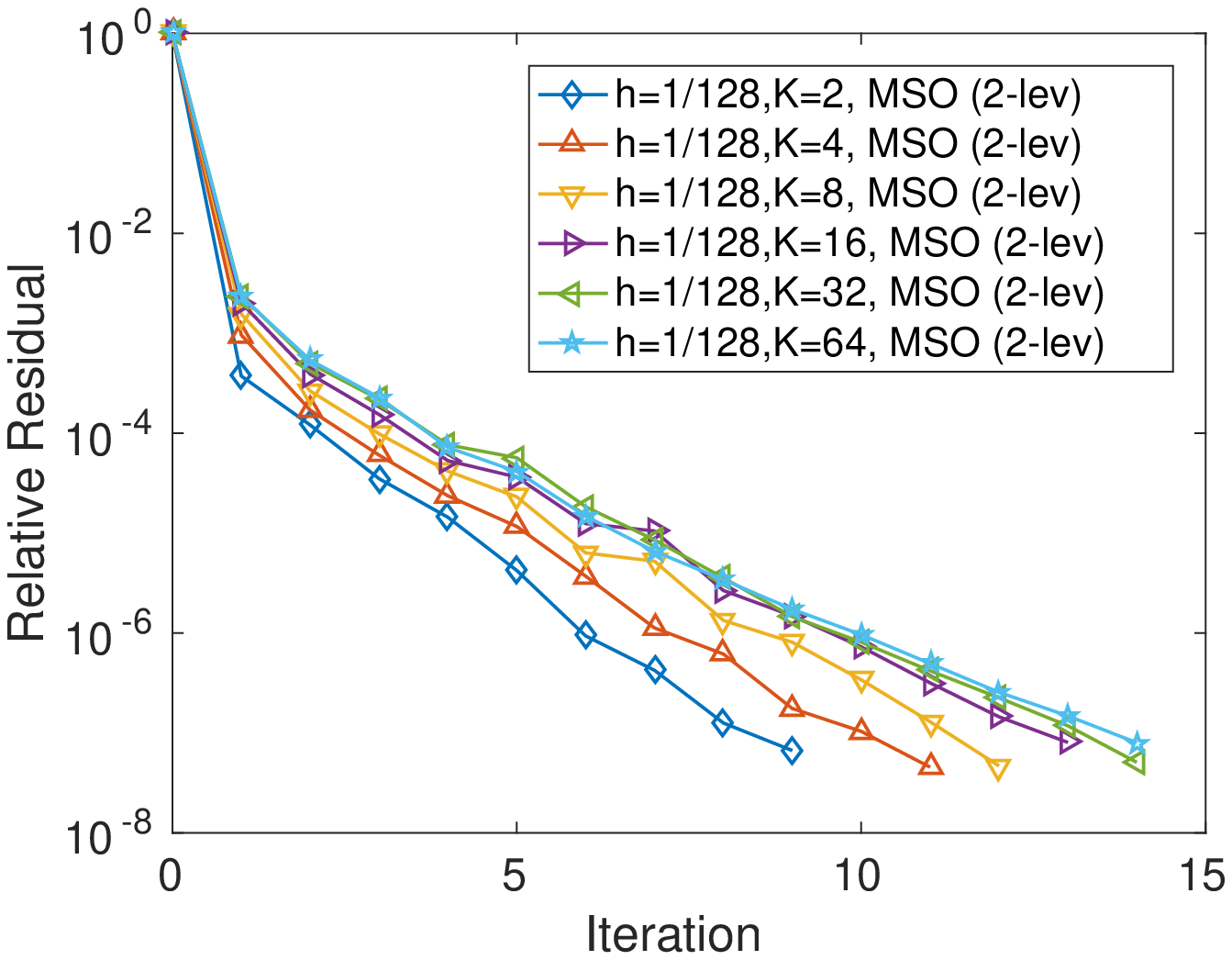}
\includegraphics[width=0.45\textwidth]{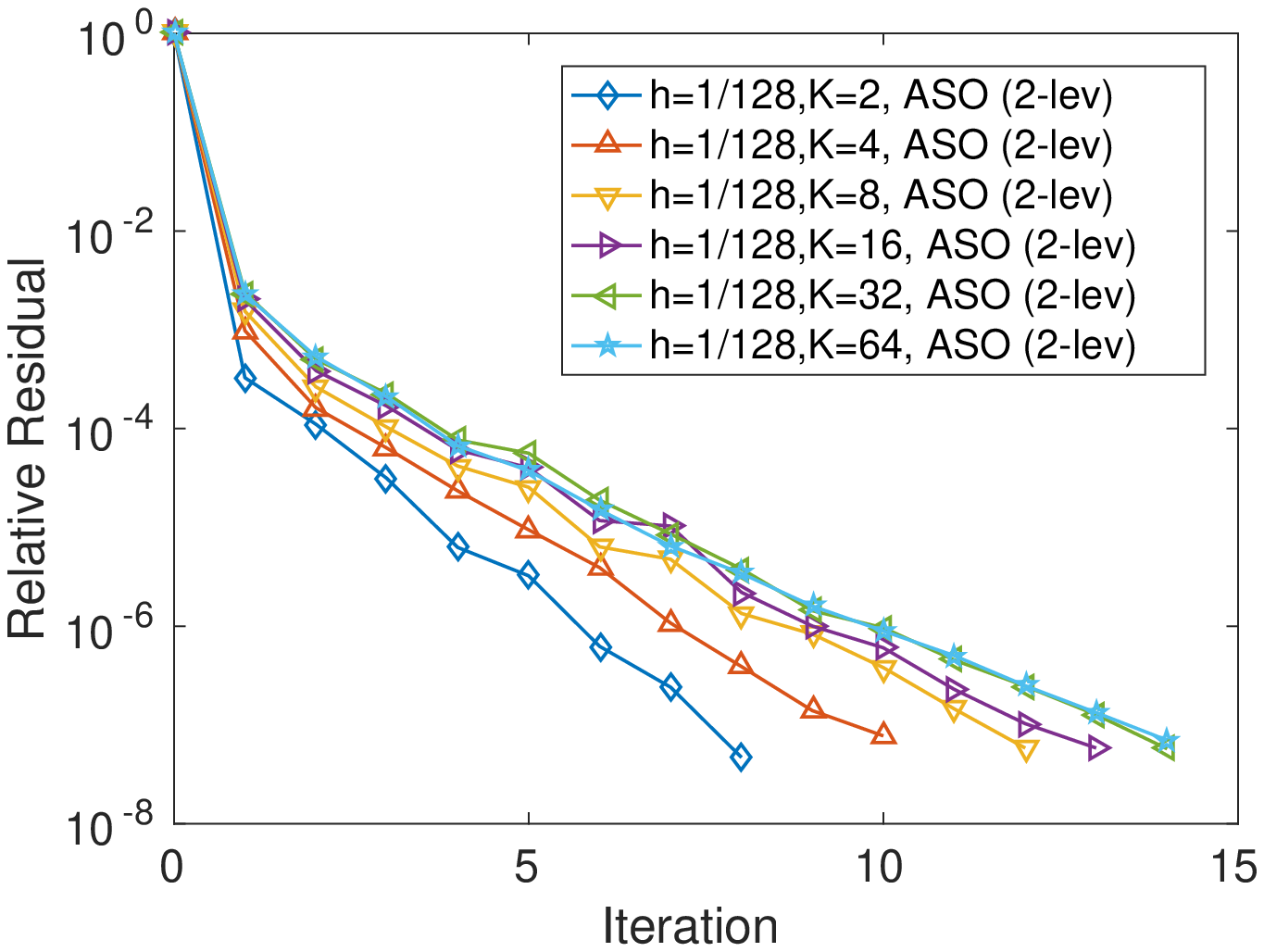} 
\caption{Convergence of GMRES with 2-level MSO and ASO preconditioners, respectively. }
\label{fig:ResMSOPre2}
\end{figure}
Based on the superior convergence performance of the above 2-level domain decomposition algorithms as stand-alone solvers,
it is reasonable to expect such 2-level algorithms work well as preconditioners for GMRES.
As shown in Figures \ref{fig:ResMSNPre2}--\ref{fig:ResMSOPre2}, 
the preconditioned GMRES with 2-level preconditioners indeed deliver much more robust convergence rate for large $K$.
It is worthwhile to point out that 2-level preconditioners based on the coarse space correction may not necessarily take less number of iterations than 1-level {preconditioners} 
when $K$ is relatively small.
This can be explained by the possible large approximation errors introduced by a very coarse grid correction.
However, 2-level preconditioners would become significantly better than 1-level preconditioners for a slightly large $K$.
{For conciseness, the corresponding CPU times are not reported here, since they are very similar to the results given in Tables \ref{T1A} and \ref{T1B}.  }
\begin{figure}[H]
\centering 
\includegraphics[width=0.45\textwidth]{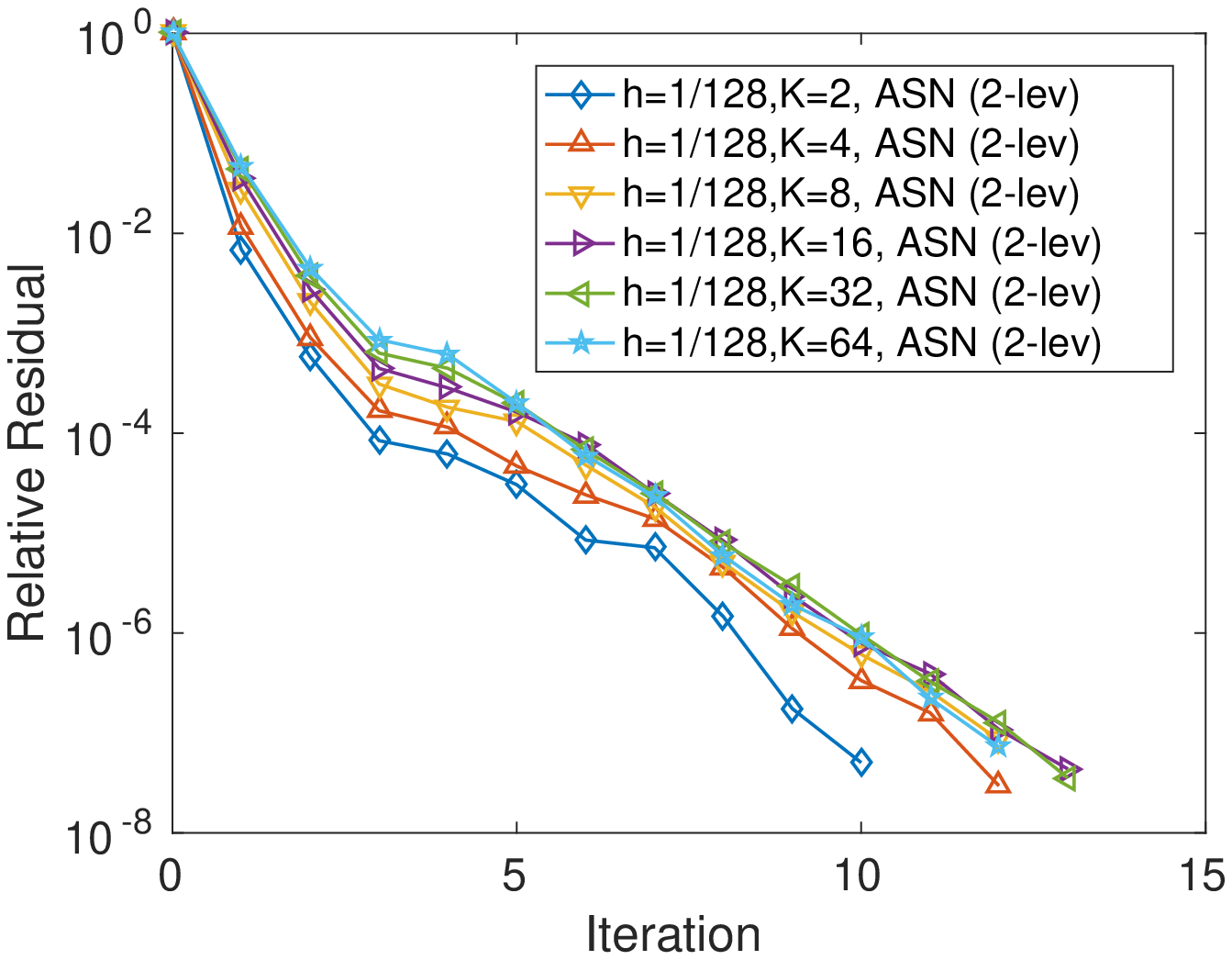}
\includegraphics[width=0.45\textwidth]{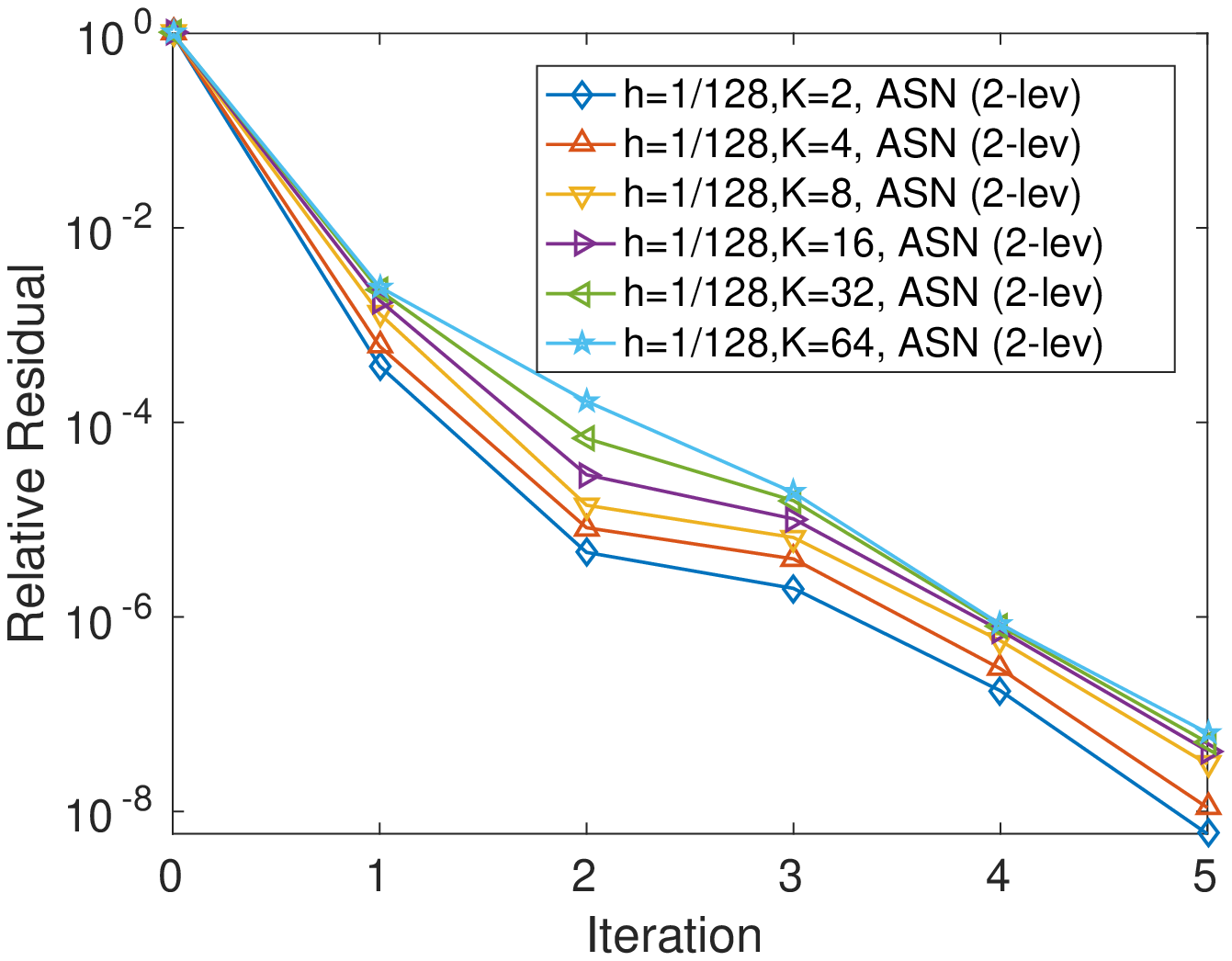} 
\caption{Convergence of GMRES with 2-level ASN preconditioners for $\gamma=10^{-4}$ (left) and  $\gamma=10^{-6}$ (right), respectively. }
\label{fig:ResASN2gam}
\end{figure}
The small parameter $\gamma>0$ is usually chosen to be close to zero,
which leads to the linear system (\ref{linsys}) with a large condition number of $O(1/\gamma)$.
Consequently, it is often of great {interest} to obtain efficient solvers with parameter-robust convergence rates.
In Figure \ref{fig:ResASN2gam}, we show the convergence of GMRES with 2-level ASN preconditioners for a much
smaller parameter $\gamma=10^{-4}$ and $\gamma=10^{-6}$, respectively.
In view of Figure \ref{fig:ResMSNPre2} with $\gamma=10^{-2}$, 
our 2-level ASN preconditioners seem to have very robust convergence rates with respect to decreasing $\gamma$.
{This complements the above shown $h$-independent and $K$-independent convergence rates.}
\subsection{2D Examples with Multigrid Subdomain Solvers} \label{num2D}
Considering that time variable is only one dimensional, the application of our proposed algorithms to 2D or 3D problems is almost the same as the 1D problems.
According to the convergence analysis in Section \ref{sec: itm}, our proposed algorithms 
are also convergent for such parabolic PDE control problems with two and three spatial dimensions.
However, in 2D or 3D cases, subdomain systems become very expensive to be solved by sparse direct methods as we did in the above 1D cases.  
Hence, more efficient iterative solvers will be used for only approximately solving subdomain systems.
More specifically, we will use the semi-coarsening multigrid solver with one V-cycle developed in \cite{LX2015a} as {the} subdomain system solver for our 2D tests.
Notice that using more accurate subdomain system solvers will lead to improved convergence {rates but may require} higher computational costs.
Based on our numerical simulations, one V-cycle multigrid iteration balances both accuracy and efficiency very well.
In addition, the coarse space correction step in 2-level methods also requires fast iterative solvers, since it essentially solves a huge sparse linear system 
from an equivalent 3D discretization with only $O(K)$ grid points in the temporal direction.
We will perform the coarse space correction by solving the coarse system approximately 
with the preconditioned bi-conjugate gradient stabilized method (BiCGStab) 
with an ILU(0) preconditioner (i.e., with 0 level of fill in), a tolerance of $10^{-4}$ and less than $200$ iterations.

\textbf{Example 2.} Let $\Omega=(0,1)^2$, $T=4$, and $\gamma=10^{-2}$.
%The parabolic optimal control problem \eqref{goal}-\eqref{state} 
%with a regularization parameter $\gamma=10^{-2}$ is considered. 
Choose $f$, $g$, and $y_0$ in (\ref{opt1A}) so that the exact solution is given by
$$y(x,t)=\cos(\pi t)\sin(\pi x_1)\sin(\pi x_2) 
 \tn{and}
p(x,t)=\sin(\pi t)\sin(\pi x_1)\sin(\pi x_2).$$

\begin{figure}[H]
\centering 
\includegraphics[width=0.45\textwidth]{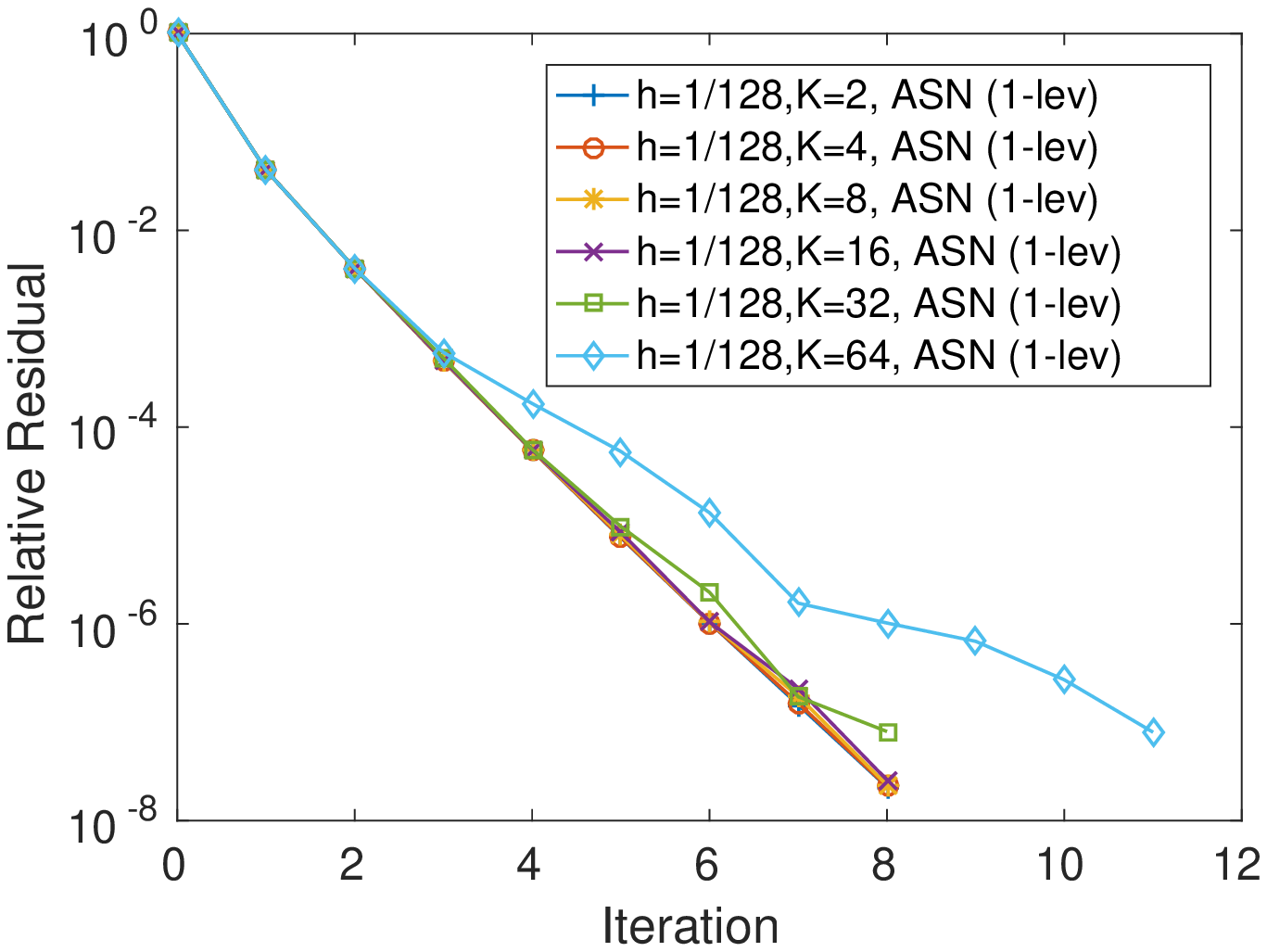}
\includegraphics[width=0.45\textwidth]{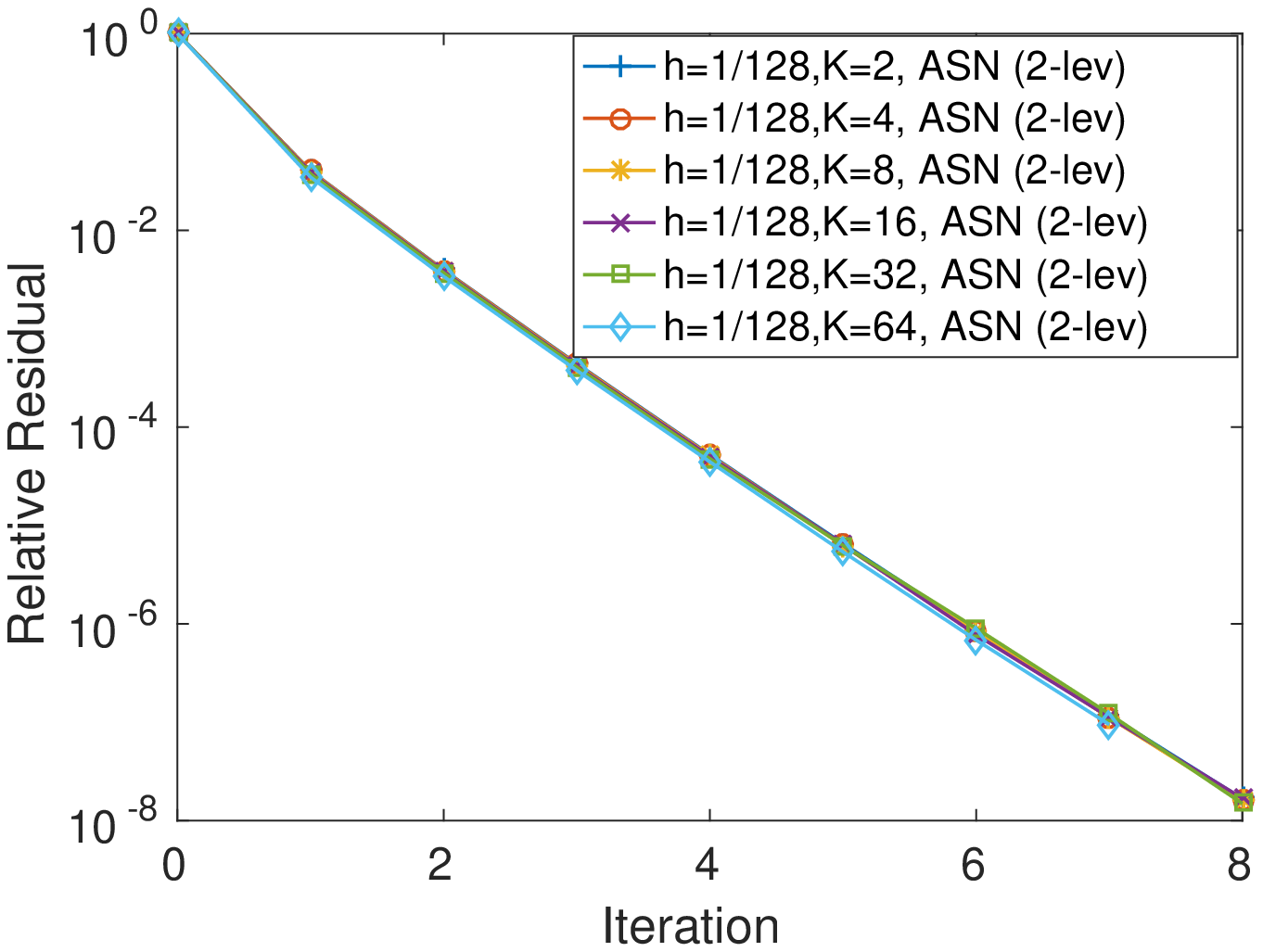}
%\end{minipage}
\caption{Convergence of 1-level and 2-level ASN as stand-alone solvers, respectively. }
\label{fig:ResASNSolver2D}
\end{figure}

\begin{figure}[H]
\centering 
\includegraphics[width=0.45\textwidth]{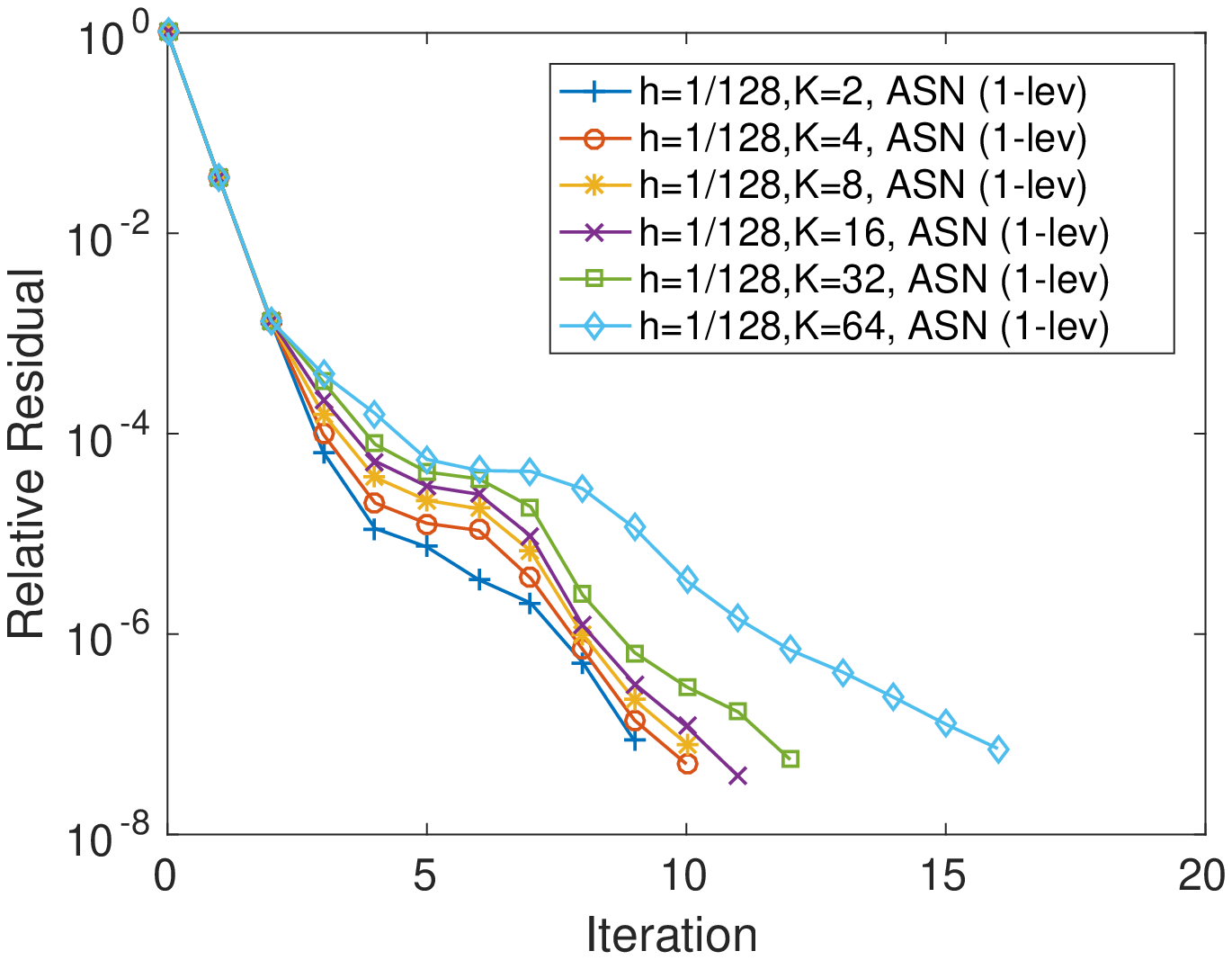}
\includegraphics[width=0.45\textwidth]{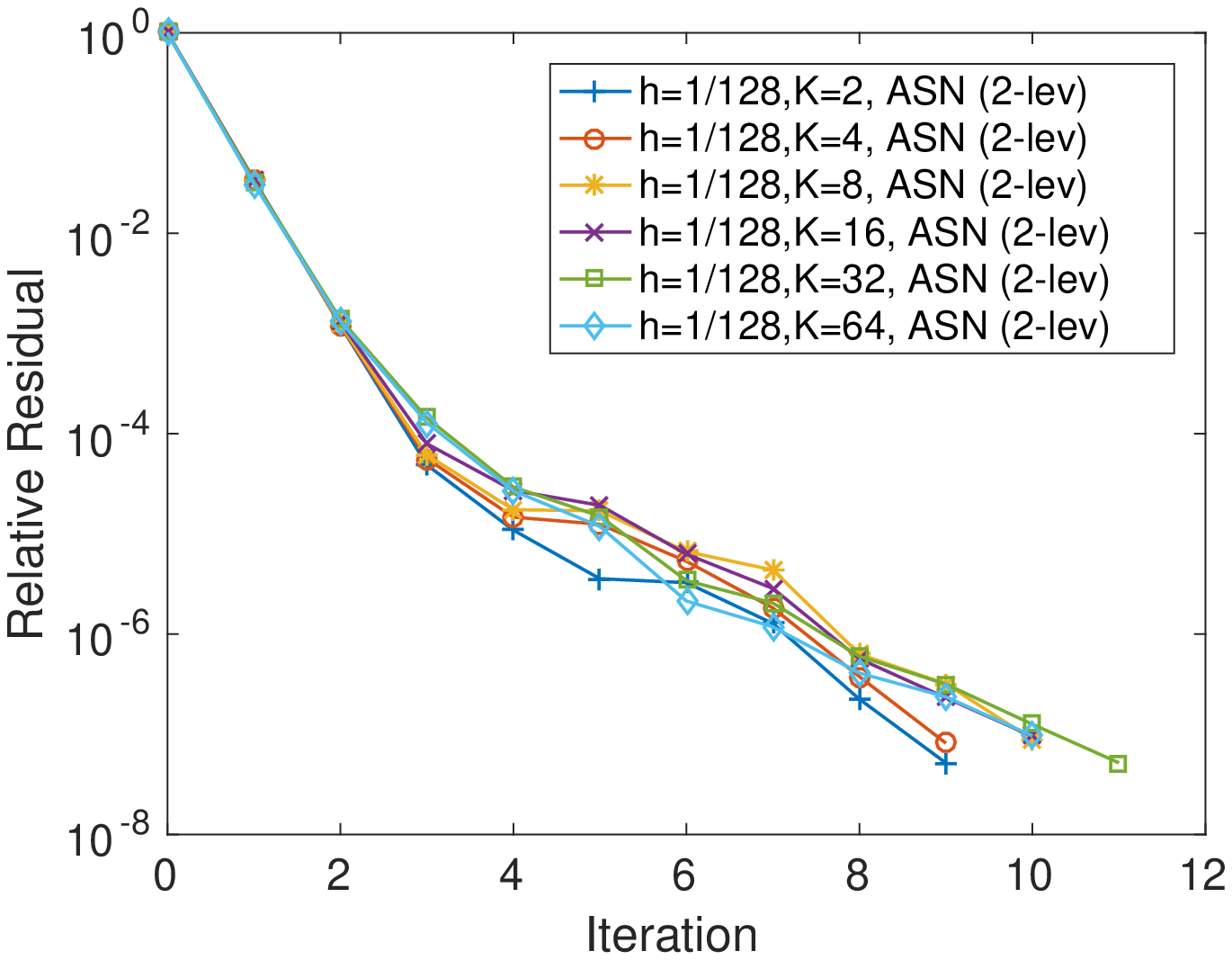}
%\end{minipage}
\caption{Convergence of GMRES with 1-level and 2-level ASN preconditioners, respectively. }
\label{fig:ResASNPre2D}
\end{figure}

In 2D cases, the convergence properties of all the proposed algorithms are essentially the same as the preceding 1D numerical results,
hence we only report the simulation results of the ASN algorithm to keep our exposition concise.
In Figure \ref{fig:ResASNSolver2D}, we show the convergence of 1-level  and 2-level ASN algorithms
as stand-alone solvers, respectively. 
In Figure \ref{fig:ResASNPre2D}, we show the convergence of right-preconditioned GMRES with 1-level and 2-level ASN algorithm as preconditioners, respectively. 
Both figures are very similar to the 1D cases, where the 2-level ASN algorithm shows {a} clearly improved convergence rate {over} the 1-level ASN algorithm
{when} $K=64$. Furthermore, the required number of iterations stay roughly the same as that in the 1D cases.  
In particular, the residual norms in 2D cases are monotonically decreasing. % (compared with Figure \ref{fig:ResMSNSolver} and Figure \ref{fig:ResMSOSolver}).
 {Notice that the 2D full discretization with $h=\tau=1/128$ gives a large sparse linear system with about $128\times 128\times 128\times 4=8,388,608$ unknowns.}
\section{Conclusions}
\label{sec:conclusions}
We proposed, analyzed, and compared several different time domain decomposition algorithms for solving linear parabolic PDE control problems. 
These algorithms aim at solving the forward-and-backward optimality PDE system with parallel computers
so that the one-shot method becomes more affordable {when} dealing with 3D large-scale, time-dependent PDE optimization problems.
The proposed algorithms are shown to be convergent as stand-alone solvers and very effective as preconditioners of Krylov subspace methods. 
To achieve a more robust convergence rate with respect to the possible large number of subdomains, 
two-level algorithms based on coarse space correction techniques are also studied in our numerical tests.
Numerical results illustrate that two-level domain decomposition algorithms perform significantly better than the corresponding one-level algorithms
when the number of subdomains becomes large. 

 The parallel implementations of our proposed algorithms on massively parallel computers are currently undertaken
and the strong and weak scaling results ({with more realistic parallel CPU times}) will be reported elsewhere. Our proposed \textit{additive} Schwarz solvers and preconditioners, ASN and ASO, are anticipated to be highly parallelizable.  
Another future work is to estimate the condition number of our proposed algorithms as preconditioners to theoretically justify the observed
 {robust convergence performance of the preconditioned GMRES}. 
{Additionally, although we have restricted our discussion to linear problems, the generalization of our proposed algorithms to nonlinear problems
can be accomplished straightforwardly, provided efficient nonlinear solvers (e.g, Newton's method, nonlinear full-approximation storage (FAS) multigrid method) for solving localized nonlinear subsystems can be developed.
Notice that time evolution has no nonlinearity.
}

 {
Upon completion of our current manuscript, we noticed some recent related work \cite{Gander_2016,Felix_2016}, where
some advanced 1-level optimized Schwarz methods based on certain parameterized Robin transmission conditions are proposed for obtaining better convergence rates, 
but their presented convergence analysis is technically more involved and 
the reported results are based on the first-order accurate backward Euler scheme in time and only limited to 1-level 
additive Schwarz nonoverlapping algorithm. 
Our current work provides more different domain decomposition algorithms (both 1-level and 2-level ones) with very comprehensive numerical comparison.
}

\section*{Acknowledgments}
The authors would like to thank the three anonymous referees for their valuable comments and detailed
suggestions that have greatly contributed to improving the presentation of the paper.
This collaboration was initiated at the Institute for Mathematics and its Applications (IMA) when J. L. and Z. W. took the IMA new directions short course -- topics in control theory in May 2014. 
We thank IMA for the travel support. 
%The IMA receives funding from the NSF under Award DMS-0931945.
%\section*{References}
 
\bibstyle{elsarticle-harv} %elsarticle-num
\bibliography{Ref_timedd,paraControl}

\end{document}